\documentclass[english, 12pt]{article}
\usepackage[T1]{fontenc}
\usepackage{geometry}
\usepackage[latin9]{inputenc}
\usepackage{amsthm}
\usepackage{amsmath}
\usepackage{amssymb}
\usepackage{authblk}
\usepackage[margin=2cm, labelfont = bf]{caption}
\usepackage{setspace}
\usepackage{esint}
\usepackage{enumerate}
\usepackage{enumitem}
\usepackage{etoolbox}
\usepackage{url}
\usepackage{graphicx, epstopdf}
\usepackage[english]{babel}
\usepackage{xcolor}
\usepackage{multicol}
\usepackage{soul}
\usepackage{tikz}
\usepackage{hyperref}
\usepackage{authblk}
\usetikzlibrary{automata,arrows,positioning,calc}

\numberwithin{equation}{section}

\makeatletter
\theoremstyle{plain}
\newtheorem{thm}{\protect\theoremname}[section]
\ifx\proof\undefined
\newenvironment{proof}[1][\protect\proofname]{\par
\normalfont\topsep6\p@\@plus6\p@\relax
\trivlist
\itemindent\parindent
\item[\hskip\labelsep
\scshape
#1]\ignorespaces
}{%
\endtrivlist\@endpefalse
}
\providecommand{\proofname}{Proof}
\fi
\theoremstyle{plain}
\newtheorem{lem}[thm]{\protect\lemmaname}
\theoremstyle{plain}
\newtheorem{prop}[thm]{\protect\propositionname}
\theoremstyle{plain}

\theoremstyle{definition}
\newtheorem{defn}[thm]{\protect\definitionname}
\theoremstyle{remark}
\newtheorem{rem}[thm]{\protect\remarkname}
\theoremstyle{plain}
\newtheorem{cor}[thm]{\protect\corollaryname}
\theoremstyle{definition}
\newtheorem{example}[thm]{\protect\examplename}
\newtheorem{assumption}[thm]{\protect\assumptionname}
\theoremstyle{plain}
\numberwithin{figure}{section}

\makeatother

\usepackage{babel}
\providecommand{\conjecturename}{Conjecture}
\providecommand{\corollaryname}{Corollary}
\providecommand{\definitionname}{Definition}
\providecommand{\examplename}{Example}
\providecommand{\lemmaname}{Lemma}
\providecommand{\propositionname}{Proposition}
\providecommand{\remarkname}{Remark}
\providecommand{\theoremname}{Theorem}
\providecommand{\assumptionname}{Assumption}

\newcommand{\R}{\ensuremath{\mathbb{R}}}

\usepackage{mathrsfs}

\begin{document}

%%%%%%%%%%%%%%%%%%%%%%%%%%%%%%%%%%%%%%%%%%%%%%%%%%%%%%%%%%%%%%%%
%%%%%%%%%%%%%%%%%%%%%%%%%%%%%%%%%%%%%%%%%%%%%%%%%%%%%%%%%%%%%%%%
%					TITLE PAGE
%%%%%%%%%%%%%%%%%%%%%%%%%%%%%%%%%%%%%%%%%%%%%%%%%%%%%%%%%%%%%%%%
%%%%%%%%%%%%%%%%%%%%%%%%%%%%%%%%%%%%%%%%%%%%%%%%%%%%%%%%%%%%%%%%

\title{Volterra clocks and their pure-jump limits: \\hitting times of curved boundaries}

\author[1]{Eduardo Abi Jaber}
\author[1]{Elie Attal}
\author[2]{Andreas S\o jmark}
\affil[1]{CMAP, \'Ecole Polytechnique}
\affil[2]{Department of Statistics, LSE}

\maketitle

\begin{abstract}
		We introduce a class of continuous Volterra processes, called Volterra clocks, and study their singular limit as the memory kernel collapses to a Dirac mass at zero. The dynamics are parametrised by a function $f$ acting as a nonlinear time-change, generalising the Volterra square-root process and recovering it when $f$ is affine. In the singular limit, the continuous Volterra clock converges weakly to a pure-jump process given by first passage times of a Brownian motion to curved boundaries, including affine and square-root boundaries when $f$ is, respectively, affine or quadratic. Outside the affine setting, characteristic function methods are no longer available, and we instead identify the limit directly from the dynamics. We do this through a topological framework adapted to the time-change structure which involves Skorokhod's $M_1$ topology and a decorated notion of convergence. Our analysis unifies several regimes of interest for general Volterra clocks, including large-time asymptotics, fast mean reversion, and hyper-roughness. In particular, this subsumes and extends existing results in the affine setting.
\end{abstract}

\section{Introduction}

Limit theorems for continuous stochastic models typically involve continuous limits. This is true even if the analysis is often cast on the Skorokhod space $D[0,T]$, reflecting the fact that the continuous functions form a closed subspace when $D[0,T]$ is equipped with the traditional $J_1$ topology. Discontinuous limits of continuous processes may, however, arise in singular regimes, as will be the case in the present work.

The singular scaling limits we obtain are of an intriguing nature: for a general class of continuous stochastic Volterra processes, we show that, as the memory kernels degenerate into a Dirac mass, the processes admit a weak limit which falls within a classical class of pure-jump processes given by the first passage times of a Brownian motion with respect to curved boundaries. Specifically, these are of the form
\begin{equation}
\label{eq:hitting_time}
 \tau_t := \inf\{s \geq 0\,:\,W_s < b(s) - G(t) \}\,, \quad t \geq 0\,,
\end{equation}
for a Brownian motion $W$, where $b$ and $G$ are nonnegative, non-decreasing functions.

For fixed $t > 0\,$, the stopping time $\tau_t$ in \eqref{eq:hitting_time} belongs to a classical family of Brownian first passage times, whose study has been the subject of a rich literature dating back to Bachelier \cite{bachelierTheorieSpeculation1900} and Schr\"odinger \cite{schrodingerZurTheorieFall1915} in the case of an affine boundary $b(s) = b_0 + b_1 s\,$. The latter leads to the Inverse Gaussian distribution. The square-root boundary $b(s) = b_0 + b_1 \sqrt{s}$ also benefits from explicit expressions, as studied in \cite{breimanFirstExitTimes1967a, sheppFirstPassageProblem1967a, novikovStoppingTimesWiener1971}. Beyond these two cases, the distribution of the first passage time to a curved boundary is generally not explicit. Thus, a substantial part of the literature has been devoted to approximation methods. These include tangent approximation \cite{strassenAlmostSureBehavior1967a, danielsMaximumSizeClosed1974a, lercheBoundaryCrossingBrownian1986, ferebeeTangentApproximationOnesided1982a} and the use of integral equations \cite{fortetFonctionsAleatoiresType1943, durbinFirstPassageDensityContinuous1985a, durbinFirstPassageDensityBrownian1992a, peskirIntegralEquationsArising2002}. While there is a wealth of results for the individual hitting times, the emergence of \eqref{eq:hitting_time} as a functional scaling limit of continuous stochastic Volterra processes provides a new perspective on these classical objects.

\subsection{Scaling limits of stochastic Volterra processes}

Our motivation originates from recent results on particular scaling limits of integrated affine stochastic Volterra processes of the form
\begin{equation}
\label{eq:affine_Volterra_intro}
X_t = \int_0^t Y_s\,ds, 
\qquad 
Y_t = g_0(t)  + \int_0^t K(t-s)\left(-\lambda Y_s\,ds + \nu\sqrt{Y_s}\,dB_s\right).
\end{equation}
Such processes arise in the study of population dynamics \cite{feller1951diffusion}, catalytic superprocesses \cite{dawson_super-brownian_1994,mytnik_uniqueness_2015}, scaling limits of Hawkes processes \cite{jaisson2016rough}, and stochastic volatility modeling \cite{abi2019affine}. In the classical Markovian setting $K\equiv 1$, \cite{fordeLARGEMATURITYSMILEHESTON} first connected these processes to the Inverse Gaussian distribution via large-time asymptotics, while \cite{mechkovFastReversionLimitHeston2014, McCrickerd2021, jaber2024reconciling} did so under a fast rescaling, a regime recently extended to fractional kernels in \cite{bondiLevyProcessesWeak2025}. In the non-Markovian setting, \cite{dawson_super-brownian_1994} obtains a general L\'evy distribution in the large-time limit for catalytic superprocesses. These limits can typically be expressed by first passage times for Brownian motion subject to affine boundaries (see \cite[Section 1.2.5]{Kyprianou}). Of particular relevance to the present work, the case of a fractional kernel $K(t)=t^{\alpha-1}$ with $\alpha < 1/2$ was recently studied in \cite{abi2025hyper} and it was shown that $X$ admits an Inverse Gaussian L\'evy process as a scaling limit when $\alpha \downarrow 0$ (this was first studied in \cite{jaber2024reconciling} via a Markovian proxy).

In the above works, the hitting time interpretations of the limits are associated with linear boundaries. Moreover, the analysis relies heavily on the affine structure and can exploit the semi-explicit form of the characteristic functional of $X\,$ to establish finite-dimensional convergence to a known limiting distribution. The work \cite{McCrickerd2021} is a notable exception: it proceeds directly from the dynamics via a time-change argument in a pathwise framework, an approach that is closer in spirit to the present paper, though it is not adapted to the Volterra setting. Beyond the affine setting, a natural question of independent mathematical interest is whether first passage processes for curved boundaries arise as the singular limits when the kernel tends to a Dirac mass in a suitable non-affine version of \eqref{eq:affine_Volterra_intro}. This would uncover a general structural link between continuous stochastic Volterra processes and pure-jump processes defined by Brownian first hitting times of curved boundaries as in \eqref{eq:hitting_time}.

One possible approach to this would be to consider \eqref{eq:affine_Volterra_intro} for non-affine stochastic Volterra equations obtained by replacing the square-root diffusion coefficient $\sqrt{Y}$ with a more general cofficient $\sigma(Y)$. This would yield a standard stochastic Volterra equation for $Y$, see for instance \cite{berger1980volterra, abi2021weaksolution}. However, such a modification does not preserve a key structural property underlying the hitting time limits in the affine case. Indeed, thanks to the stochastic Fubini theorem, the affine dynamics imply that $X$ in \eqref{eq:affine_Volterra_intro} satisfies an autonomous stochastic equation
\begin{align}\label{eq:introbarVaffine}
    X_t = G_0(t)  + \int_0^t K(t-s)\left(-\lambda X_s + \nu M_s\right)\, ds,
\end{align}
where $M$ is a continuous martingale with quadratic variation $\langle M \rangle = X$. This autonomy is central to identifying the limiting process as a Brownian first passage time.  Introducing nonlinearities directly in the dynamics of $Y$ typically destroys this property and would therefore prevent the appearance of hitting time characterisations in the limit.

\subsection{Volterra clocks and limit theorems}

The above considerations naturally lead us to study nonlinearities directly at the level of the equation \eqref{eq:introbarVaffine} for $X$.  Specifically, we introduce a class of processes which arise from allowing $M$ in  \eqref{eq:introbarVaffine}  to be a continuous martingale with quadratic variation $\langle M \rangle = f(X)$, for a non-negative, non-decreasing function $f$, with $f(x)=x$ reducing to the original setting. In our analysis, we exploit this structure as follows. By a suitable version of the Dambis--Dubins--Schwarz theorem, we can recast the equation for $X$ in the general form
\begin{equation}
\label{eq:introclock}
X_t
=
G_0(t)
+
\int_0^t
K(t-s)\,
\big(-\lambda X_s + \nu M_s \big)\,ds ,\quad M_t=W_{f(X_t)}
\end{equation}
on an enlarged probability space, where $W$ is a Brownian motion satisfying certain regularity conditions with respect to $X$. While the formulation  \eqref{eq:affine_Volterra_intro} demands a square-integrable kernel, we note that $K$ only needs to be (locally) integrable in \eqref{eq:introbarVaffine} and \eqref{eq:introclock}. Beyond this, $G_0$ is a non-negative, non-decreasing function, $\lambda \geq 0$, and $f$ is a non-negative, non-decreasing function of at most quadratic growth.

We shall refer to a non-decreasing process $X$ satisfying \eqref{eq:introclock} as a \emph{stochastic Volterra clock}. Under some minor assumptions, we establish the existence of such processes in Theorem \ref{thm:weak_existence_Volterra_clock_DDS}. We then proceed to study their singular limits, as the kernel degenerates. In particular, Theorem \ref{thm:weak_convergence} shows that we have the joint weak convergence
\begin{align}\label{eq:intro1}
(W^n, X^n)  \Rightarrow (W^*, X^*)\quad \text{as}\quad  K^n(t)\,dt \Rightarrow \delta_0(dt), 
\end{align}
with
\begin{align}\label{eq:intro2}
X_t^*
=
\inf \{ s \ge 0 : (1 + \lambda)s - \nu W^*_{f(s)} > G_0(t) \},\quad t\geq 0,
\end{align}
where $X^n$ converges weakly in Skorokhod's $M_1$ topology. That is, when the kernel of the Volterra dynamics collapses into a Dirac mass at zero, the continuous clocks $X^n$ converge to a pure-jump c\`adl\`ag process given by the first passage times of a time-changed Brownian motion to an affine boundary. Unwinding the time-change, we recover the setting that we discussed above, namely that of a standard Brownian motion crossing a curved boundary. Indeed, for $X^*_t$ in \eqref{eq:intro2}, we can observe that $f(X^*_t)$ corresponds to $\tau_t$ in \eqref{eq:hitting_time} with $  b(s) = (1 + \lambda)f^{-1}(s) / \nu$ and $G(t) = G_0(t) / \nu\,$. In particular, when $f$ is quadratic, one obtains a square-root boundary.

Our analysis poses new technical challenges and raises some interesting questions about convergence on Skorokhod space. First of all, we no longer have access to the affine machinery with the identification of explicit finite-dimensional distributions in the limit. Instead, we succeed in directly characterising the weak limit by \eqref{eq:intro2} at the process level, jointly with the Brownian motion, through a simple topological framework that exploits the time-changed structure of \eqref{eq:introclock} combined with a new martingale argument. This hinges on a suitable notion of convergence for processes time-changed by the clock $X^n$, including the martingale $M^n$ in \eqref{eq:introclock}, which furthermore leads us to study weak convergence on a decorated Skorokhod space (see Section \ref{sect:decorated_paths} and Theorem \ref{thm:weak_convergence}). We believe these ideas could be of relevance more broadly. For example, limit theorems for anomalous diffusion often involve time-changed processes converging weakly, in the uniform or $J_1$ topology, to a Brownian motion or stable L\'evy process time-changed by the continuous inverse of a stable subordinator, see \cite{Meerschaert_2014}. If the limiting time-change can have jumps (e.g., the case of subordinate Brownian motion \cite{Kim_Song_Vondracek}), then the machinery introduced in Section \ref{sect:func_conv_frame} could become relevant.

In summary, our results reveal a strong structural connection between stochastic Volterra processes and boundary crossing problems for Brownian motion with curved boundaries. Specifically, our new class of Volterra clocks are shown to be natural continuous approximations of Brownian first hitting time processes in a way that covers a broad range of singular scaling regimes, characterised by the degeneration of memory kernel $K(t)\,dt \Rightarrow \delta_0(dt)$. In particular, this allows us to treat several regimes that have been studied separately within a unified framework, including   
(i) large-time limits (Section~\ref{subsubsec:large_time}), (ii) fast mean-reversion limits (Section~\ref{subsubsec:large_mean_rev}), and
(iii) hyper-rough regimes for fractional kernels with $\alpha \downarrow 0$ (Section~\ref{subsubsec:H_to_-1/2}).
Even in the affine setting with $f(x)=x$, our results unify and extend existing results in several directions, and we obtain new large-time convergence results in both the Markovian and rough settings.

We conclude the introduction by noting that one can further clarify the structure of the stochastic Volterra clocks when the kernel $K$ is square-integrable. In that case, $X$ given by \eqref{eq:introclock} is absolutely continuous with $X_t=\int_0^t Y_s\,ds$, where
\begin{equation}\label{eq:square-integrable_case}
Y_t
=
g_0(t)
+
\int_0^t
K(t-s)
\Bigl(
-\lambda Y_s\,ds
+
\nu\sqrt{f'(X_s)Y_s}\,dB_s
\Bigr)
\end{equation}
Unless $f^\prime$ is constant (corresponding to the affine setting), this equation for $Y$ is no longer autonomous and now involves a path-dependent diffusion coefficient through $f'(X_s)$.  Such path-dependent diffusivity structures could be of independent interest for modeling purposes in several applied contexts, including population dynamics, catalytic superprocesses, and stochastic volatility models, see Remark \ref{rem:interpretation_Volterra_clock}.

\paragraph{Outline of the paper.} In Section \ref{sect:func_conv_frame} we first introduce the framework for functional convergence that we shall be working with. {Next, Section \ref{sec:main_results}} presents our main results, including the weak existence for our stochastic Volterra clocks (Theorem \ref{thm:weak_existence_Volterra_clock_DDS}), the weak convergence to Brownian first passage processes as the memory kernel degenerates (Theorem \ref{thm:weak_convergence}),   {and a refined convergence result connecting pre-limit bursts to the hitting time realisation of the jumps in the limit (Theorem \ref{thm:bursts})}. {Finally, Section \ref{subsec:applications}} covers applications of our methodology to large-time (\ref{subsubsec:large_time}), fast (\ref{subsubsec:large_mean_rev}) and hyper-rough (\ref{subsubsec:H_to_-1/2}) scalings. {The proofs of the results in Sections \ref{sect:func_conv_frame} and \ref{sec:main_results} are given in Section \ref{sec:proofs}}. Some technical results on the Dambis--Dubins--Schwarz theorem and convolution kernels and their resolvents are deferred to Appendices \ref{app:DDS} and \ref{app:resolvents}.

\section{Functional convergence framework}\label{sect:func_conv_frame}

While the $M_1$ topology is appropriate for $X^n$, it is not suitable for processes time-changed by the clock. In particular, the time-changed Brownian motion in the Volterra equation for $X^n$ cannot converge weakly in $M_1$ by \cite[Corollary 3.5]{sojmark2023weak}. Thus, we shall need a different framework that will allow us to deal effectively with the convergence of such processes. We present this here, with the proofs given in Section \ref{subsect:func_conv_frame_proofs}.

\subsection{Convergence of time-changed paths}\label{sect:time-changed_func_framework}

Let $C[0,\infty)$ denote the space of continuous paths $x : [0,\infty)\rightarrow \mathbb{R} $, and, for any given $T>0$, let $D[0,T]$ denote the space of c\`adl\`ag paths $x:[0,T]\rightarrow \mathbb{R}$. Furthermore, we introduce the space of non-negative and non-decreasing c\`adl\`ag paths
\begin{equation}
	D^\uparrow[0,T] = \{ x\in D[0,T]: 0\leq x(s)\leq x(t)\;\text{if}\;s\leq t  \}.
\end{equation}

We are interested in elements $x= \phi \circ \tau \in D[0,T]$ that arise explicitly as a time change of some $\phi \in C[0,\infty)$ by some $\tau \in D^\uparrow[0,T]$. We let it be understood that $\phi(s)=\phi(\tau(T))$ for all $s\geq \tau(T)$, as we are free to redefine $\phi$ outside the range of $\tau$. To work with such time-changed c\`adl\`ag paths, we define the space
\begin{equation}
	D^\circ[0,T] :=	\bigl\{ (\phi,\tau) :  \phi\in C[0,\infty), \,\tau \in 	D^\uparrow[0,T]  \bigr\},
\end{equation}
This yields a simple way of encoding more information about a given $x= \phi \circ \tau$ than the c\`adl\`ag path alone. For example, we have $x= \phi \circ \tau = \tilde{\phi}\circ \tau$ if $\tilde{\phi}$ equals $\phi$ outside the intervals $[\tau(t-),\tau(t)]$, for $t$ such that $\tau(t-)\neq \tau(t)$,
but we wish to view $(\phi,\tau)$ and $(\tilde{\phi},\tau)$ as two different processes (in $D^\circ[0,T]$), which share the same base path $x$, but which may `traverse' a jump from $x(t-)$ to $x(t)$ in different ways, as described by the function $\phi$ or $\tilde{\phi}$, respectively, on $[\tau(t-),\tau(t)]$.

We equip $	D^\circ[0,T]$ with the product topology, where $C[0,\infty)$ is given the topology of uniform convergence on compacts and $D^\uparrow[0,T]$ is given the $M_1$ topology. Consequently, $D^\circ[0,T]$ is a Polish space. Moreover, the space is easy to work with, as it is constructed directly from well-known spaces. In particular, we can recall that the $M_1$ topology admits
a simple characterisation in this setting. Indeed, for $\tau^n,\tau \in D^\uparrow[0,T]$, we have $\tau^n \rightarrow \tau$ in the $M_1$ topology if and only if $\tau^n(t)\rightarrow \tau(t)$ for all $t$ in a dense subset of $[0,T]$ that includes $0$ and $T$ (see \cite[Corollary 12.5.1]{whitt2002stochastic}).

Before addressing various convergence properties, we first introduce some notation. For any given $t\geq0$ and $\delta>0$, we write
\[
B(t,\delta):=[t-\delta,t+\delta]\cap [0,T].
\]
 Furthermore, for $z^n,z\in D[0,T]$, we will say that $z^n\rightarrow z$ locally uniformly on $\mathbb{I}$ if there is a co-countable subset $\mathbb{I} \subseteq[0,T]$ so that, for every $t \in \mathbb{I}$, we have
\[
\limsup_{n\rightarrow \infty}  \,\sup \{ | z^n(s)-z(r) |: s,r \in B(t,\delta) \} \rightarrow 0 \quad \text{as}\quad  \delta \rightarrow 0.
\]
When this is the case, we shall write $z^n \rightarrow_{\emph{\text{l.u.}}} z$ on $\mathbb{I}$. Finally, for any given $z\in D[0,T]$, we define the set of its discontinuity points
\[
D(z):=\{ t\in(0,T] : z(t)\neq z(t-)\},
\]
which we note is countable. The complement of such a set, or some version of this, will often play the role of the set $\mathbb{I}$ encountered above.

\begin{prop}[Locally uniform convergence]\label{prop:conv_compositions}
	Consider $x^n=\phi^n\circ \tau^n$ and $x=\phi\circ \tau$. If $(\phi^n,\tau^n) \rightarrow (\phi,\tau) $ in $D^\circ[0,T]$, then $\sup_{n\geq 1}\sup_{t\in[0,T]}|x_n(t)|<\infty$ and
	\begin{equation}\label{eq:pointwise}
x^n \rightarrow_{\text{l.u.}} x\quad\text{on} \quad \mathbb{I}:=[0,T]\!\setminus\!D(\tau)
	\end{equation}
with $x^n(T)\rightarrow x(T)$. In particular, $x^n(t)\rightarrow x(t)$ for all $t \notin  D(\tau)$ and $t=T$.
\end{prop}

In particular, if we have $x^n=\phi^n\circ \tau^n=g^n\circ \varsigma^n$ with $(\phi^n,\tau^n)\rightarrow (\phi,\tau)$ and $(g^n,\varsigma^n)\rightarrow (g,\varsigma)$ in $D^\circ[0,T]$, then Proposition \ref{prop:conv_compositions} gives that $\phi\circ \tau = g\circ \varsigma$ in $D[0,T]$, by right-continuity. So, there is uniqueness of limits at the level of the compositions in $D[0,T]$.

In \cite{jakubowski1997non}, Jakubowski introduced a new topology on the Skorokhod space $D[0,T]$ which he called the $S$-topology. While convergence of $(\phi^n,\tau^n)$ in $D^\circ[0,T]$ does \emph{not} imply convergence of the composed paths $x^n=\phi^n \circ \tau^n$ in any of the Skorokhod topologies, it does indeed imply their convergence in the $S$-topology.

\begin{prop}[$S$-convergence]\label{prop:S-conv} Let everything be as in Proposition \ref{prop:conv_compositions}. Then, we furthermore have that $x^n \rightarrow x$ in $D[0,T]$ with respect to the $S$-topology.
	\end{prop}

The convergence in $D^\circ[0,T]$ keeps track of critical information which the $S$-topology ignores. This is seen already from Proposition \ref{prop:conv_compositions}, as the marginal projections are nowhere continuous for the $S$-topology. For this work, we are interested in continuous paths $x^n=\phi^n \circ \tau^n$ converging to a discontinuous limit. In that case, we can furthermore make the following observation, which highlights a particular instance of the information that we can attach to a limiting c\`adl\`ag path $x$. We do the arguments also for $x^n$ that are nearly continuous for large $n$ in the sense of uniformly vanishing jumps.

\begin{prop}[Limiting jump range]\label{prop:cont_image_jump}
Let $d_\mathrm{H}$ be the Hausdorff distance on $\mathbb{R}$. Consider $x^n=\phi^n\circ \tau^n$ and $x=\phi\circ \tau$, and suppose $(\phi^n,\tau^n) \rightarrow (\phi,\tau)$ in $D^\circ[0,T]$ with $\sup_{t\in D(\tau^n)}\Delta \tau^n(t)\rightarrow 0$. Then, for every $t\in D(\tau)$, we have
\begin{equation}\label{eq:jump_range}
\limsup_{n\rightarrow \infty} d_{\mathrm{H}}\bigl(\hspace{1.2pt}x^n\bigl(B(t,\delta)\bigr) ,\hspace{1.2pt}\phi\bigl([\tau(t-),\tau(t)]\bigr)\bigr) \rightarrow 0\quad \text{as} \quad \delta \rightarrow 0.
\end{equation}
\end{prop}

As per Proposition \ref{prop:conv_compositions}, convergence in $D^\circ[0,T]$ with limit $x=\phi \circ \tau$ always yields
\[
\limsup_{n\rightarrow \infty} d_{\mathrm{H}}\bigl(\hspace{1.5pt}x^n\bigl(B(t,\delta)\bigr) ,\hspace{1.2pt}
\{x(t)\} \hspace{0.2pt}\bigr) \rightarrow 0\quad \text{as}\quad \delta \rightarrow 0,
\]
for $t\notin D(\tau)$. 
Note that there can be a non-trivial limiting range $\phi([\tau(t-),\tau(t)]) \neq \{ x(t)\}$ even if $t$ is a continuity point of $x$. The following observation for linear combinations of paths follows in the same way as the previous proposition.

\begin{cor}\label{cor:sum_hausdorff_range} Let $(\phi^n,\tau^n)$ and $(\psi^n,\varsigma^n)$ be as in Proposition \ref{prop:cont_image_jump}. Write $y^n=\psi^n\circ \varsigma^n$ and $y=\psi \circ \varsigma$. If $D(\tau) \cap D(\varsigma) = \emptyset$, then, for all $t\in[0,T]$, we have
	\[
\limsup_{n\rightarrow \infty} d_{\mathrm{H}}\bigl(\hspace{1.5pt}(x^n+y^n)\bigl(B(t,\delta)\bigr) ,\hspace{1.2pt}
I_t\hspace{0.2pt}\bigr) \rightarrow 0\quad \text{as}\quad \delta \rightarrow 0,
	\]
where $I_t = \{x(t)+y(t)\}$ if $t\notin D(\tau) \cup D(\varsigma)$, $I_t = \{x(t)\} + \psi([\varsigma(t-),\varsigma(t)])$ if $t\in D(\varsigma)$ and $I_t = \{y(t)\} + \phi([\tau(t-),\tau(t)])$ if $t\in D(\tau)$. 
\end{cor}

\begin{rem}\label{rem:nearly_cont_phi-n}
All of the above results readily generalise to $\phi^n \in D[0,\infty)$ with uniform convergence on compacts to $\phi \in C[0,\infty)$. This follows similarly to how the vanishing jumps of $\tau^n$ are treated in the proofs together with the continuity of the limit $\phi$.
\end{rem}

We provide two basic examples where the limit of a key functional is characterised by convergence in $D^\circ[0,T]$. In both cases, the limit cannot be deduced from just $S$-convergence of $x^n$ to $x$. Given $(\psi ,\varsigma)\in D^\circ[0,T]$, we define
\[
\underline{\psi}(s):=\!\inf_{r\in[\varsigma(0),\varsigma(0)\lor s]}\!\psi(r), \quad \overline{\psi}(s):=\!\sup_{r\in[\varsigma(0),\varsigma(0)\lor s]}\!\psi(r), \quad \text{and}\quad \hat{\psi}(s)
	:= \psi(s) - \underline{\psi}(s).
\]

\begin{cor}[Running infima and suprema]\label{cor:inf_sup} Consider $x^n=\phi^n\circ \tau^n$ with $(\phi^n,\tau^n)\rightarrow (\phi,\tau)$ in $D^\circ[0,T]$ and $\sup_{t\in D(\tau^n)}\Delta \tau^n(t)\rightarrow 0$. Then, 
$\inf_{s\leq \cdot} x^n_s = \underline{\phi}^n \circ \tau^n + \epsilon^n$ where $\sup_{t\leq T}|\epsilon^n(t)|\rightarrow 0$ and $(\underline{\phi}^n,\tau^n)\rightarrow (\underline{\phi},\tau)$ in $D^\circ[0,T]$, with $\varepsilon^n\equiv 0 $ if $\tau^n$ is continuous. The analogous statement holds for $\sup_{s\leq \cdot } x^n_s$. In particular, we have $\inf_{s\leq \cdot}x^n_s \rightarrow \underline{\phi} \circ \tau$ and $\sup_{s\leq \cdot}x^n_s \rightarrow \overline{\phi} \circ \tau$ in $(D[0,T],M_1)$.
\end{cor}

Of course, with $x=\phi \circ \tau$,  the limits $\underline{\phi} \circ \tau$ and $\overline{\phi} \circ \tau$ are in general different from $\inf_{s\leq \cdot} x_s$ and $\sup_{s\leq \cdot} x_s$. Note also that these limits can be recovered from the range $I_t = \phi([\tau(t-), \tau(t)])$ in Proposition \ref{prop:cont_image_jump} alone. The next examples needs the fact that convergence in $D^\circ[0,T]$ also keeps track of the traversal $\phi : [\tau(t-), \tau(t)] \rightarrow I_t$.

\begin{cor}[Skorokhod map]\label{cor:skorokhod}
	With everything as in
	Corollary~\ref{cor:inf_sup}, consider the Skorokhod map $	x^n \mapsto (l^n,\hat{x}^n)$ given by $l^n(t) := -\inf_{s \le t} x^n(s)$
	and $\hat{x}^n(t):= x^n(t) + l^n(t)$.
	If each $\tau^n$ is continuous, then $ l^n = -\underline{\phi}^n \circ \tau^n $ and $\hat{x}^n
	= \hat{\phi}^n \circ \tau^n$ with
	$(\underline{\phi}^n, \tau^n)
	\to (\underline{\phi}, \tau)$ and $(\hat{\phi}^n, \tau^n)
	\to (\hat{\phi}, \tau)$ in $D^\circ[0,T]$.
    If $\sup_{t \in D(\tau^n)}
	\Delta \tau^n(t) \to 0$, the same holds up to $\varepsilon^n(t)$ as in Corollary \ref{cor:inf_sup}. In particular, when $\sup_{t \in D(\tau^n)}
	\Delta \tau^n(t) \to 0$, we have $l^n \to -\underline{\phi} \circ \tau$
	in $(D[0,T], M_1)$ and $\hat{x}^n$ satisfies \eqref{eq:jump_range} with limiting ranges $I_t=\hat{\phi}([\tau(t-), \tau(t)])$.
\end{cor}

\subsection{Decorated c\`adl\`ag paths}\label{sect:decorated_paths}

In \cite[Ch.~15]{whitt2002stochastic}, the notion of `decorated' c\`adl\`ag paths (or rather c\`adl\`ag paths with `excursions') was introduced and presented as a promising area for future research. However, further work on this topic has only recently emerged. In a setting closely related to our time-change setup, \cite{McCrickerd2021} proceeds similarly to \cite{whitt2002stochastic} but constructs a different pseudometric space of interval-valued paths. In a series of works on limit theorems for dynamical systems, \cite{FFT,FFMT,CKM} generalise the spaces introduced in \cite[Ch.~15]{whitt2002stochastic}, maintaining the c\`adl\`ag structure and constructing new topologies that allow for capturing more general aspects of the decorations.

Notably, \cite{CKM} recovers a Polish space, while this aspect is missing in the analysis of \cite{whitt2002stochastic, McCrickerd2021, FFT,FFMT}. The framework of the previous section is specific to time-changed processes, but it is simpler to work with, Polishness is automatic, and, for our purposes, it in fact encodes more information about the limit. To make precise that we are also capturing metric convergence to decorated c\`adl\`ag paths, we introduce a precise definition that closely resembles the presentation in \cite[Ch.~15]{whitt2002stochastic}. 

\begin{defn}[Decorated Skorokhod space]\label{def:decorated_Skorokhod} We write $\mathfrak{D}[0,T]$ for the space of pairs $(x,I)$, where $x\in D[0,T]$ and
	$I=(I_t)_{t\in[0,T]}$ is a family of compact
sets in $\mathbb{R}$ such that, for all $t\in[0,T]$, we have $x(t-), x(t)\in I_t$ with
\begin{equation}\label{eq:I_cadlag_prop}
		d_{\mathrm{H}}(I_s,\{x(t)\})\to 0\; \text{as}\;
		s\downarrow t \quad\text{and}\quad 
		d_\mathrm{H}(I_s,\{x(t-)\}) \to 0 \; \text{as} \;
		s\uparrow t.
	\end{equation}
	We call $x$ the \emph{base} path, while $I_t$ is the \emph{decoration} of $x$ at time $t$.
\end{defn}

From the definition, one readily sees that $(x,I)\in \mathfrak{D}[0,T]$ implies compactness of the graph  
\begin{equation}\label{eq:graph}
\Gamma(I):= \bigl\{(t,z)\in[0,T]\times\mathbb{R}:z\in I_t\bigr\}.
\end{equation}
Thus, we can metrize $\mathfrak{D}[0,T]$ by setting
\begin{equation}\label{eq:metric}
d_\mathfrak{D}\bigl((x,I),(x^\prime,I^\prime)\bigr)
:=\bar{d}_{\mathrm{H}}(\Gamma(I),\Gamma(I^\prime ))
\end{equation}
where $\bar{d}_{\mathrm{H}}$ is the
Hausdorff metric on the set of non-empty compact subsets of $[0,T]\times\mathbb{R}$ induced
by $d_\infty((t,l),(t^\prime,l^\prime))=|t-t^\prime|\lor |l - l^\prime|$. The main difference from \cite[Ch.~15]{whitt2002stochastic} is that we seek a pointwise characterisation of the convergence of c\`adl\`ag paths $x^n$ to a decorated limit $(x,I)$, without filling in the `gap' in the graph from $x^n(t-)$ from $x^n(t)$. Thus, we embed $D[0,T]$ in $\mathfrak{D}[0,T]$ by the minimal choice $\iota(x^n):=\{x^n(t-),x^n(t)\}$, identifying $x^n$ with $(x^n,\iota(x^n))$. This is a continuous embedding for the $J_1$ topology, but fails to be so for the other Skorokhod topologies and the $S$ topology.

\begin{prop}[Characterisation of decorated limits]\label{prop:pointwise_char_decorated} Consider $x^n \in D[0,T]$ and $(x,I)\in \mathfrak{D}[0,T]$. We have $x^n\rightarrow (x,I)$ in $\mathfrak{D}[0,T]$ if and only if
\begin{equation}\label{eq:decorated_char}
	\limsup_{n\rightarrow \infty} d_{\mathrm{H}}\bigl(\hspace{1.5pt}x^n\bigl(B(t,\delta)\bigr) ,\hspace{1.2pt}
	I_t \hspace{0.2pt}\bigr) \rightarrow 0\quad \text{as}\quad \delta \rightarrow 0,\quad \text{for all} \quad t\in[0,T].
\end{equation}
\end{prop}

In the definition of $\mathfrak{D}[0,T]$, one can of course also consider $\mathbb{R}^d$-valued base paths $x$ with $I_t \subset \mathbb{R}^d$. The proof of Proposition \ref{prop:pointwise_char_decorated} extends verbatim to that setting.

Beyond \eqref{eq:metric}, particular topologies for decorated c\`adl\`ag paths (as in \cite[Ch.~15]{whitt2002stochastic} and \cite{FFT,FFMT,CKM}) amount to encoding specific information about the decorations, such as the order in which each decoration $I_t$ is traversed. While $\mathfrak{D}[0,T]$ can be considered the starting point, decorations must be distinguished accordingly, thus leading to refined spaces. For our analysis, the convergence of $(\phi^n,\tau^n)$ in $D^\circ[0,T]$ not only gives convergence of $x^n=\phi^n \circ \tau^n$ in $\mathfrak{D}[0,T]$ due to \eqref{eq:decorated_char}, it also specifies the exact traversal of $I_t=\phi([\tau(t-),\tau(t)])$ by the function $\phi$ on $[\tau(t-),\tau(t)]$ which is more information than any of the decorated topologies encode.

\section{Volterra clocks: existence and limit theorems}
\label{sec:main_results}

We now return to our analysis of Volterra clocks. As discussed in the introduction, we shall study weak existence and establish certain limit theorems, using the framework of the previous section. To further motivate this class of processes, we first provide some context about the affine case.
A Volterra square-root process $Y$ is a continuous nonnegative stochastic process, solving the stochastic Volterra equation 
\begin{equation}
\label{eq:Volterra_sqrt_intro}
Y_t = g_0(t) + \int_0^t\,K(t-s)\,\left(- \lambda \, Y_s + \nu\,\sqrt{Y_s}\,dB_s \right)\,, 
\end{equation}
where $B$ is a standard Brownian motion, and $K \in L^2_{loc}(\R_+\,, \R_+)$ is a convolution kernel. Such equations can be viewed as non-Markovian extensions of classical square-root diffusions, in which the history of the process influences its evolution through the memory kernel $K\,$. When the kernel is exponential, i.e.~of the form $ce^{-\lambda t}\,$, with $c ,\lambda \geq 0\,$, the process $Y$ is Markovian. However, this is no longer true for a general memory kernel. If the kernel is singular at zero, as it is the case for the fractional kernel $K(t) = c t^{\alpha - 1}\,$, with $c \geq 0$ and $1/2 < \alpha < 1\,$, then $Y$ is not even a semimartingale.

Processes with square-root noise term arise naturally in scaling limits of branching systems. The classical example is the Feller diffusion (in the Markovian case $K \equiv 1$), obtained as the diffusive limit of critical Galton-Watson processes. More generally, consider one-dimensional particles evolving as Markov processes with generator $\mathcal L$ and branching only in the presence of a catalyst located at the origin. In the critical branching regime, under appropriate rescaling, one obtains a limiting \textit{single-point catalytic superprocess} whose density $Y_t$ formally satisfies 
$$
\dot{Y_t} = \mathcal L Y_t + \delta_0\,\sqrt{Y_t}\,dB_t\,.
$$
If the semigroup associated with $\mathcal L$ corresponds to a convolution with a density $p_t(\cdot)\,$, the mild formulation at $x = 0$ reads
\begin{equation}
\label{eq:catalytic_Volterra_sqrt}
Y_t(0) = \int_{\R}\,p_t(-y)\,Y_0(y)\,dy + \int_0^t\,p_{t-s}(0)\,\sqrt{Y_s(0)}\,dB_s\,,
\end{equation}
so that $Y_t(0)\,$, the particle density at the catalyst location, satisfies a Volterra square-root equation with kernel $t \mapsto p_t(0)\,$. These processes have been widely studied in the case of Brownian particles, leading to the \textit{single-point catalytic super-Brownian motion} of \cite{dawson_super-brownian_1994}, see also \cite{fleischmann1995new, dawson1995singularity}, and \cite{zahle_space-time_2005, mytnik_uniqueness_2015} for the SPDE viewpoint. More general single-point catalytic processes have also been considered, such as symmetric $\alpha$-stable particles with generator $\mathcal L = - (-\Delta)^{\alpha / 2}\,$, $1 < \alpha \leq 2\,$, or stricly elliptic diffusions, i.e.~$\mathcal L = \mu(\cdot)\cdot \nabla + \frac{1}{2}\sigma^2(\cdot)\cdot \Delta\,$, see  \cite{dawson_critical_1991,klenke_absolute_2000, wang_fluctuation_2014}, leading to various memory kernels.

Volterra square-root processes are also widely used in mathematical finance. In particular, equation \eqref{eq:Volterra_sqrt_intro} is satisfied by the variance in Volterra Heston stochastic volatility models, see \cite{el2019characteristic, abi2019affine, abi2021weak}. In this context, they have been obtained as scaling limits of nearly unstable Hawkes processes in \cite{jaisson2016rough, jusselin2020no}.
\\

In several examples the kernel $K$ is not square-integrable, but only in $L^1_{loc}(\R_+\,, \R_+)\,$. This is the case for the single-point catalytic super-Brownian motion, where $p_t(0) \propto t^{-1/2}$ in \eqref{eq:catalytic_Volterra_sqrt}. The stochastic convolution in \eqref{eq:Volterra_sqrt_intro} is then no longer defined and the process $V$ cannot be constructed directly. Nevertheless, the integrated process $X_t := \int_0^t\,Y_s\,ds$ remains well-defined and, fortunately, still captures the affine structure as it satisfies the autonomous equation \eqref{eq:introbarVaffine}.
\\

The class of stochastic Volterra clocks we introduce in this section aims to generalise this autonomous equation by introducing a time-change function $f\,$, leading to
\begin{equation}
\label{eq:Volterra_clock}
X_t = G_0(t) + \int_0^t\,K(t-s) \,\left(-\lambda X_s + \nu\,W_{f(X_s)}\right)\,ds\,.
\end{equation}
We establish weak existence of a non-decreasing solution $X$ to \eqref{eq:Volterra_clock} under suitable assumptions in Section \ref{subsec:new_processes}, and weak convergence to Brownian first crossing pure-jump processes in Section \ref{subsec:weak_convergence}.

\subsection{Weak existence}
\label{subsec:new_processes}
Before stating our weak existence results, we introduce the main assumptions and definitions used throughout the paper. In particular, we mainly focus on \textit{completely monotone} memory kernels. 

\begin{defn}
\label{defn:cm}
    $K : \R_+ \to \R_+$ is said to be completely monotone if it is infinitely differentiable on $(0\,, + \infty)\,$, and 
    $$
    (-1)^n\,K^{(n)} \geq 0\,, \quad n \geq 0\,.
    $$
\end{defn}
\noindent
This is broad class which contains exponentially decaying kernels $K(t) = c\,e^{-\lambda t}\,$, with $c, \lambda \geq 0\,$, fractional kernels $K(t) = c \,t^{\alpha - 1}\,$, with $c \geq 0$ and $\alpha \leq 1\,$, and gamma kernels $K(t) = c\,e^{-\lambda t}\,t^{\alpha - 1}\,$, with $c\,, \lambda \geq 0$ and $\alpha \leq 1\,$. For a given kernel $K \in L^1_{loc}(\R_+\,, \R_+)\,$, we define a set of admissible input curves $g_0\,$, for which we will be able to prove existence results.
\begin{defn}
    \label{defn:input_curves}
    For $K \in L^1_{loc}(\R_+\,, \R_+)$ we define $\mathcal G_{K}$ as the set of function $g_0 : \R_+ \to \R_+\,$, of the form 
    $$
    g_0(t) = a(t) + \int_0^t K(t-s)\,b(s)\,ds\,, \quad t\geq 0\,,
    $$
    where $a\,, b: \R_+ \to \R_+\,$, $a$ is non-decreasing and $b$ is locally bounded, measurable.
\end{defn}
\noindent
Finally, we will require that our time-change function $f$ satisfies the following assumption.
\begin{assumption}
    \label{assumption:f}
    $f \in \mathcal C^1(\R_+\,, \R_+)\,$, is non-decreasing, starts from zero, and its derivative $f'$ has at most linear growth, i.e.~$|f'(x)| \leq c(1 + x)\,$, $x \geq 0\,$, for some $c \geq 0\,$.
\end{assumption}
\noindent In particular, under Assumption \ref{assumption:f}, $f$ has at most quadratic growth.\\

The previous assumptions are precisely those needed to obtain the following weak existence result.
\begin{thm}
    \label{thm:weak_existence_Volterra_clock}
    Let $\lambda\,, \nu \geq 0\,, T > 0\,$. We consider $K \in L^1_{loc}(\R_+\,, \R)$ a completely monotone kernel, $g_0 \in \mathcal G_K\,$, and $f : \R_+ \to \R_+$ satisfying Assumption \ref{assumption:f}. Then, there exists a filtered probability space $\left(\Omega\,, \mathcal F\,, \left(\mathcal F_t\right)_{t \leq T}\,, \mathbb P\right)$ satisfying the usual conditions and adapted continuous processes $\left(X\,, M\right)\,$, starting from $0\,$, such that $X$ is non-decreasing, $M$ is a martingale with quadratic variation $\langle M\rangle_t = f(X_t)\,$, and 
    \begin{equation}
    \label{eq:dynamic_bar_V}
    X_t = G_0(t) + \int_0^t\,K(t-s)\,\left( - \lambda\,X_s + \nu\,M_s\right)\,ds \,, \quad t \leq T\,,
    \end{equation}
    where $G_0 := \int_0^{\cdot}\,g_0(s)\,ds\,$.
\end{thm}
\begin{proof}
    The proof is given in Section \ref{subsec:proof_existence}.
\end{proof}
Compared to the affine Volterra framework (see \cite{abi2019affine, abi2021weak}), where $f$ is a linear function, we allow $f$ to be up to quadratic (since its derivative has linear growth from Assumption \ref{assumption:f}). The well-posedness of this extension can be understood as follows. Applying the BDG inequality to the martingale $M\,$, we have 
$$
\mathbb E \left[\sup_{s \leq t}\,|M_s|^p \right] \leq C_p\,\mathbb E \left[f(X_t)^{p / 2}\right]
$$
for any $p \geq 1\,$, hence we can expect to obtain closed Volterra-type inequalities on the moments $X$ via \eqref{eq:dynamic_bar_V}, if the growth of $f$ is up to quadratic. To the best of our knowledge, this margin with respect to the affine case has not yet been exploited in the literature. Furthermore, it appears that estimates on $X$ become difficult to obtain whenever $f$ is more than quadratic, and weak existence of a solution to \eqref{eq:dynamic_bar_V} is unlikely.

In the case where $K \in L^2_{loc}(\mathbb R_+\,, \R_+)\,$, as discussed in the introduction, we can apply the stochastic Fubini theorem to obtain that $X_t = \int_0^t\,Y_s\,ds\,$, where $Y \geq 0$ solves the stochastic Volterra equation \eqref{eq:square-integrable_case}. Similarly to the affine case, when the kernel $K$ is no longer locally square-integrable, nothing guarantees that $X$ admits a density with respect to the Lebesgue measure, and only \eqref{eq:dynamic_bar_V} is well-defined (see \cite{dawson_super-brownian_1994, jaisson2016rough, abi2021weak}).
\begin{rem}
\label{rem:interpretation_Volterra_clock}
    The extension \eqref{eq:square-integrable_case} admits several natural interpretations. In the context of branching or Hawkes-type processes, it corresponds to a branching intensity depending on the accumulated activity $X_t$. For instance, one may model catalytic mechanisms with fatigue or self-excitation depending on cumulative occupation time. In population dynamics, it can represent environmental feedback depending on the total accumulated population. In mathematical finance, when $Y$ represents the variance of a stock price, the factor $\sqrt{f'(X_t)}$ captures the idea that past volatility spikes (corresponding to crises) may amplify future fluctuations.
\end{rem}

\begin{rem}
    Note that \eqref{eq:square-integrable_case} is different from the general stochastic Volterra equation
    $$
    Y_t = g_0(t) + \int_0^t\,K(t-s)\,\left(- \lambda\, Y_s\,ds + \sigma(Y_s)\,dW_s\right)\,,
    $$
    in that it contains the non-Markovian factor $f'(X_s)\,$, where $X := \int_0^{\cdot}\,Y_s\,ds\,$.
\end{rem}

In view of the functional convergence framework developed in Section \ref{sect:func_conv_frame}, and of the hitting time limits we aim to obtain, we want to exploit the specific form of the quadratic variation of $M$ in Theorem \ref{thm:weak_existence_Volterra_clock} to represent it as a time-changed Brownian motion $W_{f(X)}$ via the Dambis--Dubins--Schwarz theorem (specifically Proposition \ref{prop:DDS}), leading to \eqref{eq:Volterra_clock}.

\begin{defn}
\label{defn:weak_sol_Volterra clock}
    We say that $(W\,, X)$ is a weak solution to \eqref{eq:Volterra_clock} on $[0\,,T]$ for the input $(G_0\,, K\,, f\,,\lambda\,, \nu)$ if there exists a filtered probability space $\left( \Omega\,, \mathcal F\,, (\mathcal G_t)_{t \geq 0}\,, \mathbb P\right)$ supporting a $(\mathcal G_t)_{t \geq 0}$-Brownian motion $W\,$, and a non-decreasing continuous process $(X_t)_{t \leq T}\,$, starting from $0\,$, such that:
\begin{enumerate}[label=(\roman*)]
    \item For every $t\in[0,T]\,$, $f(X_t)$ is a $(\mathcal G_t)_{t \geq 0}$-stopping time.
    \item $(W_{f(X_t)})_{t\leq T}$ is a martingale with respect to its own filtration.
    \item The Volterra equation \eqref{eq:Volterra_clock} holds $\mathbb P$-almost surely on $[0\,,T]\,$.
\end{enumerate}
\end{defn}
 
In the setup of Theorem \ref{thm:weak_existence_Volterra_clock}, we can make use of the Dambis--Dubins-Schwarz representation given by Proposition \ref{prop:DDS} with $(X\,, M\,, A) := (X\,, M\,, f(X))$ to obtain the following weak existence result.
\begin{thm}
    \label{thm:weak_existence_Volterra_clock_DDS}
    Let $\lambda\,$, $\nu\,$, $T\,$, $K\,$, $G_0$ and $f$ be as in Theorem \ref{thm:weak_existence_Volterra_clock}. Then, there exists a weak solution $(W\,, X)$ to \eqref{eq:Volterra_clock} on $[0\,,T]$ for the input $(G_0\,, K\,,f\,, \lambda\,, \nu)\,$, in the sense of Definition \ref{defn:weak_sol_Volterra clock}.
\end{thm}

\subsection{Weak convergence to Brownian first crossing jump processes}
\label{subsec:weak_convergence}
By introducing the time-change $f(X)$ in the Volterra clock dynamics \eqref{eq:Volterra_clock}, we expect to obtain a broader class of limiting Brownian boundary-crossing problems. Recall that the limit we study in the present work corresponds to the degeneracy of the convolution kernel $K$ to the Dirac measure $\delta_0(dt)\,$. We will see in Section \ref{subsec:applications} that this framework jointly describes fast, large-time and hyper-rough scalings of Volterra clocks. In particular, it yields new convergence results even in the affine case $f(x) = x\,$.
\\

Formally replacing $K(\cdot)\,ds$ by a Dirac mass at zero in \eqref{eq:Volterra_clock} leads to the pointwise equation
\begin{equation}
\label{eq:hitting_time_equation}
X^*_t = G_0(t) - \lambda \,X^*_t + \nu\,W^*_{f(X^*_t)}\,, \quad t \leq T\,.
\end{equation}
We first show that this is indeed true in the limit, and we then employ a martingale argument to establish that $s=X^*_t$ is in fact the first time that the trajectory $s\mapsto (1+ \lambda) \,s - \nu\,W^*_{f(s)}$ crosses above $G_0(t)\,$, giving
\begin{equation}
\label{eq:first_hitting_time_equation}
X^*_t = \inf\{ s \geq 0\,:\, (1 + \lambda) s - \nu W^*_{f(s)} > G_0(t)\}\,, \quad t \leq T\,.
\end{equation}
Both parts rely, in different ways, on the functional framework of Section \ref{sect:func_conv_frame}. In the affine case $f(x) = x\,$ with $G_0(t)=ct$, $X^*$ is an Inverse Gaussian L\'evy process which was obtained as the limit of specific models in \cite{jaber2024reconciling, abi2025hyper}. \\

Consider a sequence of weak solutions $(W^n,X^n)$ for given kernels $K^n$. The conjectured weak convergence cannot be studied in Skorokhod's classical $J_1$ topology, as it does not allow continuous functions to tend to discontinuous limits. Since the clocks $X^n$ are non-decreasing, the convergence is naturally handled by the weaker $M_1$ topology (recall Section \ref{sect:time-changed_func_framework}). However, passing to the limit in the Volterra equation also requires us to address the weak convergence of the martingale part $M^n_t = W^n_{f^n(X^n_t)}\,$ which is handled natively within the framework for $D^\circ[0,T]$ developed in Section \ref{sect:time-changed_func_framework}.  In the affine setting of \cite{abi2025hyper}, the convergence of $M^n$ was instead studied in the weaker $S$-topology, and the limit of $X^n$ was identified via the characteristic functional, which is known semi-explicitly thanks to the affine structure. In the present setting, the characteristic functional is no longer available, but we succeed in identifying the limit directly from the dynamics. Exploiting the time-changed structure of $M^n$ as an element of $D^\circ[0,T]$, we prove that it converges to $M^* = W^*_{f(X^*)}\,$ and establish that $M^*$ is a martingale, a fact which is pivotal for obtaining \eqref{eq:first_hitting_time_equation}. The martingale property of $M^*$ is confirmed in Proposition \ref{prop:martingality_limit} via the precise control offered by $D^\circ[0,T]$, while it was only obtained from uniform integrability via $L^2$-control in \cite{abi2025hyper} which we do not have here. In fact, $X^*_t$ need not have finite second order moment in the case of a square-root boundary; see in particular \cite{novikovStoppingTimesWiener1971}, corresponding to $f$ quadratic.\\

Based on the above ideas, we arrive at the following functional limit theorem for the Volterra clocks $X^n$ when sending the kernel $K$ to a Dirac mass at the origin.
\begin{thm}
    \label{thm:weak_convergence}
    Let $\left(W^n\,, X^n\right)_{n \geq 0}$ be a sequence of weak solutions to \eqref{eq:Volterra_clock} on $[0\,,T]\,$, for the sequence of inputs $(G_0^n\,, K^n\,, f^n\,, \lambda^n\,, \nu^n)_{n \geq 0}\,$, where $\lambda^n\,, \nu^n \geq 0\,$, $f^n$ satisfies Assumption \ref{assumption:f}, $K^n \in L^1_{loc}(\R_+\,, \R_+)$ and $G_0^n : \R_+ \to \R_+$ is non-decreasing with $G_0^n(0) = 0\,$, for each $n \geq 0\,$. Suppose that:
    \begin{itemize}
        \item $K^n(t)\,dt \Rightarrow \delta_0(dt)$ weakly on $[0\,,T]\,$.
        \item $G_0^n \to G_0$ pointwise on $[0\,,T]\,$, where $G_0$ is continuous.
        \item $f^n \to_{l.u.} f$ on $\R_+$, with $\sup_{n \geq 0} \sup_{x \geq 0} f^n(x) / (1 + x^2) < + \infty\,$.
        \item $\lambda^n \to \lambda$ and $\nu^n \to \nu\,$.
    \end{itemize}
    Then, we have the weak convergence of processes 
    $$
    \left(W^n\,, X^n \right) \overset{n \to \infty}{\implies} \left(W^*\,, X^*\right)
    $$
    on $D^{\circ}[0\,,T]\,$, where $W^*$ is a standard Brownian motion and 
    $$
    X^*_t = \inf\{s \geq 0\,:\, (1+ \lambda)s - \nu\,W^*_{f(s)} > G_0(t)\}\,, \quad t \leq T\,.
    $$
\end{thm}
\begin{proof}
    The proof is given in Section \ref{subsec:proof_convergence}.
\end{proof}
\begin{rem}
    \label{rem:f_fixed}
    Note that in most cases, we will keep $\lambda^n \,, \nu^n$ and $f^n$ independent of $n\,$, since the present work focuses on the convergence $K^n(t)dt \implies \delta_0(dt)\,$. However, we will use the possibility to make $\nu^n$ and $f^n$ depend on $n$ when studying large-time limits of Volterra clocks in Section \ref{subsubsec:large_time}.
\end{rem}

Keeping in mind that the Brownian motion is almost surely nowhere increasing, the limiting process $X^*$ is a pure-jump process. For each $t \geq 0\,$,  $X_t^*$ corresponds to a first hitting time of a time-changed Brownian motion to a linear boundary. More precisely, denoting by $f^{-1}$ the generalised inverse of $f\,$, we have
$$
f(X^*_t) = \inf\{s \geq 0\,:\, \nu\,W^*_s < (1 + \lambda)\,f^{-1}(s) - G_0(t)\}\,, \quad t \leq T\,.
$$
If $f$ is linear, we recover the Inverse Gaussian limit present in the literature, and even an Inverse Gaussian L\'evy process if $G_0$ is also affine. When $f$ is quadratic, $f(X_t)$ is the first hitting time of a Brownian motion to a square-root boundary.

\subsection{From pre-limit bursts to hitting time jumps}\label{sect:decorated_conv}

Theorem \ref{thm:weak_convergence} does not address how the limiting
characterisation as a first hitting time relates to the dynamics
of $X^n$. Specifically, the $M_1$ convergence reveals that steep ascents of $X^n$, over vanishing time windows, become
jumps of $X^*$, but the monotonicity of the
clock hides the fact that these ascents are driven by a
fluctuating impetus from the Brownian input, no longer seen in
the paths of $X^*$. Nevertheless, we can confirm that, in a
limiting sense, the geometric content of the hitting time
characterisation of $X^*$ is exactly capturing how the jumps arise from particular Brownian fluctuations that are amplified into bursts by the degenerating kernel. To make this precise, we write
\begin{equation}\label{eq:clock-general}
	X^n_t = G_0^n(t) + \int_0^t K^n(t-u)\,\Upsilon^n_u\,du,
	\qquad \Upsilon^n_u := \nu^n W^n_{f^n(X^n_u)} - \lambda^n X^n_u,
\end{equation}
where we single out $\Upsilon^n$ as the driving force of the
clock's increase.

\begin{rem}
	Note that $\Upsilon^n$ can be computed from the paths of $X^n$ (given $G_0^n$ and $K^n$) by convolving $X^n-G_0^n$ with the resolvent of the first kind of $K^n$ \cite[Definition 5.5.1]{gripenberg1990volterra}. This amounts to taking a generalised fractional derivative, which for fractional kernels $K^n(t)= t^{\alpha_n - 1}/\Gamma(\alpha_n)$ is the usual Caputo derivative $\Upsilon^n=\mathbb{D}^{\alpha_n}(X^n-G_0^n)$. 
\end{rem}

For large $n$, with $K^n$ concentrated near $0$, if
$\Upsilon^n_u$ is roughly constant around $t$, it follows
from~\eqref{eq:clock-general} that $X^n_u \approx G_0^n(u) +
\Upsilon^n_u$ is roughly constant and so the limit
$X^*$ should remain continuous at $t$. On the other hand, if $X^n_u \approx G_0^n(u) + \Upsilon^n_u$ is violated by $\Upsilon^n_u > X^n_u - G_0^n(u)$ on a small window of time around $t$, comparable to the support of $K^n$, then the
convolution in \eqref{eq:clock-general} picks up the surplus and
$X^n$ feels a burst that can advance it by a non-negligible amount even as the time window vanishes. This burst ends by $X^n - G^n_0$ catching up with $\Upsilon^n$, so $\Upsilon^n$ forms a `bridge' above $X^n_u - G_0^n(u)$ from $\Upsilon^n_{l_n} \approx X^n_{l_n} - G_0^n(l_n)$ to $\Upsilon^n_{r_n} \approx X^n_{r_n} -
G_0^n(r_n)$ over the burst window $(l_n,r_n)$.

In physical time, the limit of $\Upsilon^n$ will just be the finite variation process $\Upsilon^*=X^* - G_0$ and the aforementioned bridge collapses to a jump from $\Upsilon^*_{t-}$ to $\Upsilon^*_{t}$. However, we can recover the fluctuating bridge in clock coordinates. Indeed, writing 
\[
\Upsilon^n
= F^n \circ X^n,\qquad F^n(s) := \nu^n
W^n_{f^n(s)} - \lambda^n s,\qquad F^*(s) = \nu W^*_{f(s)} - \lambda s, 
\]
we get $F^*$ as the limiting bridge in clock-time with  $ F^*(s) > s-G_0(t)$ for $s\in (X^*_{t-}, X^*_t)$, where there is equality at both endpoints and $F^*$ only goes below the line immediately past $s=X^*_t$. That is,
\[
X^*_t= \inf\{s \geq 0\,:\,  F^*(s) < s - G_0(t)\},
\]
where $F^*$ on $[X^*_{t-},X^*_t]$ is now given a precise meaning as the limiting description of the fluctuation that caused a burst in the dynamics of $X^n$ and produced $\Delta X^*_t>0$ in the limit. The following formalises this as a decorated limit theorem.

\begin{thm}\label{thm:bursts} Let everything be as in Theorem \ref{thm:weak_convergence}. The weak convergence of the processes $\Upsilon^n= F^n \circ X^n$ defined in \eqref{eq:clock-general} is characterised by
\[
(F^n,X^n) \overset{n \to \infty}{\implies}  (F^*,X^*)\quad \text{on}\quad D^\circ[0,T],
\]
for $F^*$ and $X^*$ as above. In particular, we have
\[
\Upsilon^n \overset{n \to \infty}{\implies} (X^*-G_0,I^*)\quad \text{on} \quad \mathfrak{D}[0,T]
\]
with decorations $I^*_t = F^*([X^*_{t-},X^*_t])$, for $t\in D(X^*)$, and $I^*_t =\{ X^*_t-G_0(t) \}$, for $t\notin D(X^*)$.
\end{thm}
\begin{proof}The proof is given in Section \ref{subsect:rates_proof}.
\end{proof}

While $\Upsilon^n$ converges on $\mathfrak{D}[0,T]$, the convergence fails on $D[0,T]$ in any of the Skorokhod topologies. One can apply Proposition \ref{prop:S-conv} to assert $\Upsilon^n \Rightarrow \Upsilon^* = X^*-G_0$ on $D[0,T]$ in the $S$-topology, but this ignores the decoration $F^*$ on $[X^*_{t-},X^*_{t}]$.

The case of exponential kernels $K^n(t)=c_ne^{-\beta_n t}$ tending to a Dirac mass provides a particularly instructive example, as we obtain the nice rate equation
\[
\frac{d}{dt} \bigl(X^n_t - G^n_0(t) \bigr) = c_n \Bigl( \Upsilon^n_t - \frac{\beta_n}{c_n}\bigl(X^n_t - G^n_0(t) \bigr) \Bigr)
\]
with $c_n\rightarrow \infty$ and $\beta_n/c_n \rightarrow 1$. Thus, the burst structure in terms of $\Upsilon^n$, for large $n$, is directly visible in the rate of the clock. Note also that we can write $\frac{1}{c_n}\frac{d}{dt}(X^n-G_0^n)= \widetilde{F}^n \circ X^n + \frac{\beta_n}{c_n}G_0^n$ with $\widetilde{F}^n(s)=\nu^n W^n_{f^n(s)} - (1+\lambda^n)s$. So, as in Theorem \ref{thm:weak_convergence}, the convergence is characterised by $(\widetilde{F}^n,X^n)\Rightarrow (\widetilde{F}^*,X^*)$ on $D^\circ[0,T]$ with $\widetilde{F}^*(s)=\nu W^*_{f(s)} - (1+\lambda)s$. In particular, by Corollary \ref{cor:sum_hausdorff_range}, the rescaled rate $\frac{1}{c_n}\frac{d}{dt}(X^n-G_0^n)$ converges weakly to $(0,\widetilde{I}^*)$ on $\mathfrak{D}[0,T]$ with $\widetilde{I}^*_t=G_0(t)+\widetilde{F}^*([X^*_{t-},X^*_{t}])$ for $t\in D(X^*)$.

\section{Applications of the convergence result}
\label{subsec:applications}
We aim to demonstrate that the limiting regimes captured by $K^n(t)dt \Rightarrow \delta_0(dt)$ are in fact very general. In particular, this can explain the Inverse Gaussian scaling limits of affine Volterra processes that have been addressed in the literature. Our approach is more complete, as it includes a broader class of processes and associated limits. Furthermore, we provide functional characterisations of the convergence at the process level, while many existing results are at the level of the marginals. 

\subsection{Large-time limit of Volterra clocks}
\label{subsubsec:large_time}
In this section, we classify the large-time limit of Volterra clocks by applying Theorem \ref{thm:weak_convergence} to the sequence of kernels $(t \mapsto nK(nt))_{n \geq 0}\,$. As we now confirm, when the kernel is integrable on $\R_+\,$, this sequence converges to a Dirac mass at zero.
\begin{lem}
    \label{lemma:weak_cv_large_time_L1} 
    Let $K \in L^1(\R_+)\,$. We have the weak convergence of measures 
    $$
    nK(nt)\,dt \overset{n \to + \infty}{\implies} \|K\|_{L^1(\R_+)}\,\delta_0(dt)\,,
    $$
    on $[0\,,T]\,$, for any $T > 0\,$.
\end{lem}
\begin{proof}
    The proof is given in Appendix \ref{app:resolvents}.
\end{proof}

Our approach yields new large-time convergence results, both in the Markovian and rough case, and even in the affine setting $f(x) = x\,$.

We fix $\lambda \,, \nu\geq 0\,$, $f$ satisfying Assumption \ref{assumption:f}, $K \in L^1_{loc}(\R_+ \,, \R_+)$ and $G_0 : \R_+ \to \R_+$ non-decreasing with $G_0(0) = 0$, such that there exists a weak solution $(X^{(T)}, W^{(T)})$  to
$$
X^{(T)}_t = G_0(t) + \int_0^t \,K(t-s)\,\left(-\lambda\, X_s^{(T)} + \nu\,W^{(T)}_{f(X^{(T)}_s)}\right)\,ds \,, \quad t \leq T\,,
$$
for the input $(G_0, K, f, \lambda, \nu)\,$, for any $T > 0$ (see Definition \ref{defn:weak_sol_Volterra clock} and Theorem \ref{thm:weak_existence_Volterra_clock_DDS}). We want to investigate the large-time behavior of this process. Hence, we fix $T > 0\,$, define the accelerated process $X^n_t := X^{(nT)}_{nt}$ for all $t \in [0\,, T]\,$, and aim to study its limit as $n$ goes to infinity. We also denote $W^{(nT)}$ by $W^n\,$.

When the memory kernel $K$ is integrable on $\R_+\,$, we obtain the following theorem.
\begin{thm}
    \label{thm:large_time_limit_L1}
    Assume $K \in L^1(\R_+\,,\R_+)\,$, with $\|K \|_{L^1(\R_+)} = 1$ for simplicity. Then the following hold:
    \begin{itemize}
        \item If $G_0(t) \to G_0^* \in \R_+$ as $t \to + \infty\,$, then $(W^n \,, X^n) \Rightarrow (W^*\,, X^*)$ in $D^{\circ}[0\,,T]\,$, where $W^*$ is a standard Brownian motion, and $X^*$ is the constant process
        $$
        X^*_t := \inf\{ s \geq 0\,:\, (1 + \lambda)s - \nu \,W^*_{f(s)} > G_0^*\,\}, \quad t \leq T\,.
        $$
        \item If there exists a sequence $c_n \downarrow 0$ such that, defining
        $$
        G_0^n(t) := c_n\,G_0(nt)\,, \quad f^n(x) := c_n^2\,f(x / c_n)\,, \quad \widetilde{W}^n_t := c_n\,W^n_{t / c_n^2}\,, \quad t\,, x \geq 0\,,
        $$
        we have $G_0^n \to G_0^*$ pointwise on $[0\,,T]$ for some continuous $G_0^*\,$, and $f^n \to_{l.u.} f^*$ on $\R_+\,$, then $(\widetilde{W}^n\,, c_n\,X^n) \Rightarrow (W^*\,, X^*)$ in $D^{\circ}[0\,,T]\,$, where $W^*$ is a standard Brownian motion, and 
        $$
        X^*_t := \inf\{s \geq 0\,:\, (1 + \lambda)s - \nu W^*_{f^*(s)} > G_0^*(t) \}\,, \quad t \leq T\,.
        $$
    \end{itemize}
\end{thm}
\begin{proof}
    After a change of variable, we have 
    $$
    X^n_t = G_0(nt) + \int_0^t\,nK(n(t-s))\,\left( - \lambda \, X^n_s + \nu\,W^n_{f(X^n_s)}\right)\,ds\,, \quad t \leq T\,.
    $$
    From Lemma \ref{lemma:weak_cv_large_time_L1}, we know that the measure $nK(nt)\,dt$ converges weakly to $\delta_0(dt)$ on $[0\,,T]\,$, as $n \to + \infty\,$. Hence, if $G_0(nt) \to G_0^*\,$, we can apply Theorem \ref{thm:weak_convergence} to obtain the first result.

    In the case where we need to rescale $G_0(nt)\,$, we have 
    \begin{align*}
    c_n X^n_t &= G_0^n(t) + \int_0^t\,n\,K(n(t-s))\,\left(- \lambda c_n X^n_s + \nu\,c_n\,W^n_{f(X^n_s)}\right)\,ds\\
    &= G_0^n(t) + \int_0^t \,n\, K(n(t-s))\,\left( - \lambda c_n X^n_s + \nu \widetilde{W}^n_{f^n(c_n X^n_s)}\right)\,ds\,, \quad t \leq T\,.
    \end{align*}
     Since the derivative of $f$ has at most linear growth and $f(0) = 0\,$, there exists $a \geq 0\,$, such that 
    $$
    f(x) \leq a(x^2 + x)\,, \,\, x \geq 0  \quad \implies \quad f^n(x) \leq a(x^2 + c_n x) \,, \,\, x \geq 0\,, \, \, n \geq 0\,.
    $$
    Recalling that $c_n \downarrow 0$ in the limit, $\sup_{n \geq 0} \, \sup_{x \geq 0}\, f^n(x) / (1 + x^2) < + \infty\,$. 
    We can thus apply Theorem \ref{thm:weak_convergence} to the sequence $(\widetilde{W}^n\,, c_n X^n)\,$.
\end{proof}
\begin{rem}
    \label{rem:limit_rescaled_f}
    Note that $f^*(x) := \lim_{n \to + \infty} c_n^2 f(x /c_n)\,$, when it is well-defined, is the quadratic part of $f\,$. Indeed, $c_n^2 (x / c_n)^2 = x^2\,$, and if $g(x) = o_{x \to + \infty}(x^2)\,$, then $c_n^2 g(x / c_n) \to 0\,$. Hence, only a quadratic function can survive this rescaling, any other case will lead to $f^* \equiv 0\,$, i.e.~the deterministic limit $X_t^* = G_0^*(t) / (1 + \lambda)\,$.
\end{rem}

Theorem \ref{thm:large_time_limit_L1} allows to characterise the large-time limit of a large class of well studied integrated Volterra square-root processes, such as the classical Markovian square-root process, as detailed in the following example.

\begin{example}
    \label{example:large_time_limits}
    The square-root process 
    $$
    Y_t = Y_0 + \int_0^t\,\lambda \,(\theta - Y_s)\,ds + \nu \,\int_0^t\,\sqrt{Y_s}\,dB_s\,,
    $$
    with $\lambda > 0\,$, can be rewritten after a variation of constant and integration, as 
    $$
    X_t = \frac{Y_0 - \theta}{\lambda}\,(1 - e^{-\lambda t}) + \theta \,t + \nu\,\int_0^t\,e^{-\lambda (t-s)}\,W_{X_s}\,ds\,,
    $$
    where $X := \int_0^{\cdot}\,Y_s\,ds\,$. Applying Theorem \ref{thm:large_time_limit_L1} with $K(t) = e^{-\lambda t}\,$:
    \begin{itemize}
        \item If $\theta = 0\,$, $(X_{nt})_{n \geq 0}$ converges weakly, for any $t \geq 0\,$, to an Inverse Gaussian random variable of the form 
        $$
        X^* := \inf \{s \geq 0\,:\, \lambda s - \nu W^*_s > Y_0\}\,.
        $$
        To our knowledge, this convergence result is new; the closest related work \cite{fordeLARGEMATURITYSMILEHESTON} concerns large-time large deviations for the log-price in the Heston model, assuming $\theta > 0\,$.
        \item If $\theta > 0\,$, then $(t \mapsto X_{nt} / n)_{n \geq 0}$ converges weakly to $t \mapsto \theta t$ in the $M_1$ topology, which also follows from the ergodicity of $Y\,$.
    \end{itemize}
    If instead we consider the quadratic Volterra clock
    $$
    Y_t = Y_0 + \int_0^t\,\lambda\,(\theta - Y_s)\,ds + \nu\,\int_0^t\,\sqrt{X_s\,Y_s}\,dB_s\,,
    $$
    with $\theta > 0\,$, yielding the integrated dynamic
    $$
    X_t = \frac{Y_0 - \theta}{\lambda}\,(1 - e^{-\lambda t}) + \theta \,t + \nu\,\int_0^t\,e^{-\lambda(t-s)}\,W_{X_s^2}\,ds\,,
    $$
    then Theorem \ref{thm:large_time_limit_L1} shows that we obtain a non trivial first passage jump limit. More precisely, $(t \mapsto X_{nt} / n)_{n \geq 0}$ converges weakly in the $M_1$ topology to $X^*$ of the form
    $$
    X^*_t := \inf\{s \geq 0\,:\,\lambda s - \nu W^*_{s^2} > \lambda \theta t \}\,.
    $$
\end{example}

We now study non-integrable kernels, i.e.~$\int_0^{\infty} \,K(s)\,ds = + \infty\,$. In that case, with the additional assumption that $K$ is completely monotone, we show in Proposition \ref{prop:nonnegative_resolvent_cm} that its resolvent of the second kind is integrable on $\R_+\,$. Hence, we can obtain similar results by slightly modifying the equation. Denoting by $R_{\lambda}$ the resolvent of $\lambda K$ (see Appendix \ref{app:resolvents}), we have 
\begin{equation}
\label{eq:Volterra_clock_resolvent}
X^{(T)}_t = \widetilde{G_0}(t) + \frac{\nu}{\lambda}\,\int_0^t\,R_{\lambda}(t-s)\,W^{(T)}_{f(X^{(T)}_s)}\,ds\,,
\end{equation}
where $\widetilde{G_0}(t) := G_0(t) - \int_0^t  R_{\lambda}(t-s)G_0(s)ds\,$.

\begin{thm}
    \label{thm:large_time_limit_non_L1}
    Assume $K \in L^1_{loc}(\R_+\,,\R_+)$ is completely monotone with $\int_0^{\infty}\,K(s)\,ds = + \infty\,$, and $\lambda > 0\,$. Then the following hold:
    \begin{enumerate}
        \item If $\widetilde{G_0}(t) \to G_0^* \in \R_+$ as $t \to + \infty\,$, then $(W^n \,, X^n) \Rightarrow (W^*\,, X^*)$ in $D^{\circ}[0\,,T]\,$, where $W^*$ is standard Brownian motion, and $X^*$ is the constant process
        $$
        X^*_t := \inf\{s \geq 0 \,:\, \lambda s - \nu W^*_{f(s)} > G_0^* \}\,, \quad t \leq T\,.
        $$
        \item If there exists a sequence $c_n \downarrow 0$ such that, defining
        $$
        \widetilde{G_0}^n(t) := c_n\,\widetilde{G_0}(nt)\,, \quad f^n(x) := c_n^2\,f(x / c_n)\,, \quad \widetilde{W}^n_t := c_n\,W^n_{t / c_n^2}\,, \quad t\,,x \geq 0\,,
        $$
        we have $\widetilde{G_0}^n \to G_0^*$ pointwise on $[0\,,T]$ for some continuous $G_0^*\,$, and $f^n \to_{l.u.} f^*$ on $\R_+\,$, then $(\widetilde{W}^n\,, c_n\,X^n) \Rightarrow (W^*\,, X^*)$ in $D^{\circ}[0\,,T]\,$, where $W^*$ is a standard Brownian motion, and
        $$
        X^*_t := \inf\{s \geq 0\,:\, \lambda s - \nu W^*_{f^*(s)} > G_0^*(t) \}\,, \quad t \leq T\,.
        $$
    \end{enumerate}
\end{thm}
\begin{proof}
    From Proposition \ref{prop:nonnegative_resolvent_cm}, we know that $R_{\lambda}$ is nonnegative and integrable on $\R_+\,$. Moreover, recalling that $W^{(T)}_{f(X_s^{(T)})}$ is a martingale, taking expectations in \eqref{eq:Volterra_clock_resolvent} gives $\widetilde{G_0}(t) = \mathbb E [X_t^{(T)}]\,$, which is nonnegative, non-decreasing, and starts from zero. We can thus apply Theorem \ref{thm:large_time_limit_L1}, which gives the result.
\end{proof}

\begin{example}
    For the fractional square-root process
    $$
    X_t = G_0(t) + \frac{1}{\Gamma(\alpha)}\int_0^t \,(t-s)^{\alpha - 1}\,\left( - \lambda X_s + \nu \,W_{X_s}\right)\,ds\,,
    $$
     with $0 < \alpha < 1$ and $\lambda > 0\,$, the resolvent of $\lambda K\,$, given by $R_{\lambda}(t) = \lambda t^{\alpha - 1}\,E_{\alpha, \alpha}(-\lambda t^{\alpha})$, where $E_{\alpha, \beta}$ is the two-parameter Mittag-Leffler function, is integrable on $\R_+$ thanks to Proposition \ref{prop:nonnegative_resolvent_cm}. For $G_0 := \int_0^{\cdot} (K * b)(s)\,ds\,$, $b \in L^1(\R_+\,, \R_+)\,$, one verifies that $\widetilde{G_0} := G_0 - R_{\lambda} * G_0 \to G_0^* \in \R_+\,$ as $t \to + \infty\,$. Hence, $X_t$ converges weakly to 
     $$
     X^* := \inf\{s \geq 0\,:\, \lambda s - \nu W^*_{s} > G_0^* \}\,,
     $$
     which is an Inverse Gaussian random variable. While \cite{friesen2024volterra} establishes the existence of limiting distributions for $Y_t := \dot{X_t}$ when it exists (i.e.~$\alpha > 1/2$) via the characteristic functional, the present convergence for $X_t$ is new, fully explicit, and valid for any $0 < \alpha < 1\,$.

    If $\lambda = 0\,$, we can no longer apply Theorem \ref{thm:large_time_limit_non_L1}. However, formally letting $\lambda \to 0$ in the limiting equations yields 
    $$
    X^* = \inf\{s \geq 0\,:\, - \nu W^*_s > G_0^* \}
    $$
    which is a  $1/2$--stable L\'evy random variable with infinite mean. In the special case $\alpha = 1/2\,$, this corresponds to the large-time limit established in \cite{dawson_super-brownian_1994} for the total occupation time of the single-point catalytic super-Brownian motion (see \eqref{eq:catalytic_Volterra_sqrt}).
\end{example}

\subsection{Fast regime}
\label{subsubsec:large_mean_rev}
In this section, we study the fast regime of Volterra clocks, which corresponds to sending $\lambda$ and $\nu$ to infinity simultaneously in \eqref{eq:Volterra_clock}. This can be done using Theorem \ref{thm:weak_convergence}, thanks to the following result.
\begin{lem}
\label{lemma:weak_cv_resolvent_mean_rev}
    Let $K \in L^1_{loc}(\R_+\,, \R)$ be a completely monotone kernel, and $R^n$ the resolvent of $n K\,$. Then we have the weak convergence of measures
    $$
    R^n(t)\,dt \overset{n \to + \infty}{\implies} \delta_0(dt)\,,
    $$
    on $[0\,,T]\,$, for any $T > 0\,$.
\end{lem}
\begin{proof}
    The proof is given in Appendix \ref{app:resolvents}.
\end{proof}

The following theorem unifies and extends the scaling limits of affine Volterra processes converging to Inverse Gaussian distributions in \cite{mechkovFastReversionLimitHeston2014, McCrickerd2021} (Markovian case) and \cite{bondiLevyProcessesWeak2025} (fractional kernel) beyond the affine setting, while being applicable to a broader class of kernels.

\begin{thm}
    \label{thm:weak_cv_mean_rev_process}
    We consider $K \in L^1_{loc}(\mathbb R_+\,, \mathbb R_+)$ a completely monotone kernel, $\lambda\,, \nu\,, T > 0\,$, and $f$ satisfying Assumption \ref{assumption:f}. Let $a\,,b : \R_+ \to \R_+\,$, with $a$ non-decreasing and $b$ locally bounded and measurable, and define 
    $$
    G_0^n(t) := \int_0^t\,g_0^n(s)\,ds\,, \quad g_0^{n}(t) := a(t) + n\,\int_0^t\,K(t-s)\,b(s)\,ds\,, \quad t \geq 0\,,
    $$
    i.e.~$g_0^n \in \mathcal G_{nK}\,$. For each $n \geq 0\,$, let $(W^n\,, X^n)$ be a weak solution to \eqref{eq:Volterra_clock} on $[0\,,T]$ for the input $(K\,, G_0^n\,, f\,, n \lambda\,, n \nu)\,$, which exists by Theorem \ref{thm:weak_existence_Volterra_clock_DDS}. Then, as $n \to + \infty\,$, $(W^n \,, X^n) \Rightarrow (W^*\,, X^*)$ in $D^{\circ}[0\,, T]\,$, where $W^*$ is a standard Brownian motion, and 
    $$
    X^{*}_t := \inf \{s \geq 0\,:\, \lambda\,s - \nu\,W^*_{f(s)} > \bar b(t) \}\,, \quad t \leq T\,,
    $$
    with $\bar b := \int_0^{\cdot}\,b(s)\,ds\,$.
\end{thm}
\begin{proof}
    Let $n > 0\,$, we denote by $R^{n}$ the resolvent of $n \lambda K$ (see Appendix \ref{app:resolvents}), so that 
    $$
    X^{n}_t = \widetilde{G_0}^{n}(t) + \frac{\nu}{\lambda}\,\int_0^t\,R^{n}(t-s)\,W^{n}_{f(X^{n}_s)}\,ds\,, \quad t \leq T\,,
    $$
    with $\widetilde{G_0}^{n} := G_0^{n} - R^{n} * G_0^{n}\,$.
    From Proposition \ref{prop:nonnegative_resolvent_cm} and Lemma \ref{lemma:weak_cv_resolvent_mean_rev}, we know that $R^{n} \in L^1_{loc}(\R_+\,, \R_+)\,$, and 
    $$
    R^{n}(t)\,dt \overset{n \to + \infty}{\implies} \delta_0(dt)
    $$
    on $[0\,,T]\,$.
    Moreover, as argued in the proof of Theorem \ref{thm:large_time_limit_non_L1}, $\widetilde{G_0}^{n}$ is non-decreasing, nonnegative, and starts from zero. Using the definitions of $G_0^{n}$ and $R^{n}\,$, we can rewrite it as 
    $$
    \widetilde{G_0}^{n}(t) = \int_0^t\,a(s)\,ds - \int_0^t\,R^{n}(t-s)\,\int_0^s\,a(r)\,dr\,ds + \frac{1}{\lambda}\,\int_0^t\,R^{n}(t-s)\,\int_0^s\,b(r)\,dr\,ds\,, \quad t \leq T\,.
    $$
    The weak convergence of $R^{n}(t)\,dt$ to the measure $\delta_0(dt)$ shows that 
    $$
    \widetilde{G_0}^{n}(t) \overset{n \to + \infty}{\longrightarrow} \frac{1}{\lambda}\,\int_0^t\,b(s)\,ds\,, \quad t \leq T\,,
    $$
    which is continuous.

    Hence, all the conditions of Theorem \ref{thm:weak_convergence} are satisfied, so that $(W^n \,, X^n) \Rightarrow (W^*\,, X^*)$ in $D^{\circ}[0\,,T]\,$, where
    $$
X^*_t := \inf\{s \geq 0\,:\, s - \frac{\nu}{\lambda} W^*_{f(s)} > \frac{1}{\lambda}\,\bar b(t) \}\,, \quad t \leq T\,,
    $$
    which concludes the proof.
\end{proof}

In Figure \ref{fig:convergence_fast_mean_rev}, we illustrate the convergence from Theorem \ref{thm:weak_cv_mean_rev_process}, with 
$K(t) = t^{-1/2}\,$, $a \equiv 1\,$, $b(t) = 100\,e^{-t}\,$, $T = 0.1\,$, and $\lambda = \nu = 1\,$. We compare two time-change functions, $f(x) = x$ and $f(x) = x + \frac{1}{2}x^2\,$. We can clearly observe the convergence to the expected limiting jump process as $n \to + \infty\,$, in both cases. The simulation was achieved using the iVi scheme from \cite{jaber2025simulating}, which is based on an Inverse Gaussian approximation of the increments of $X\,$. In the non-affine case $f(x) = x + \frac{1}{2}x^2\,$, we perform a tangent approximation at each step, which is a classical strategy from the literature on Brownian first crossing of curved boundaries \cite{strassenAlmostSureBehavior1967a, danielsMaximumSizeClosed1974a, ferebeeTangentApproximationOnesided1982a}. In our case, it corresponds to approximating the boundary by its tangent at the hitting time, so that we recover the Inverse Gaussian distribution as a proxy.

\begin{figure}[h]
    \centering
    \includegraphics[scale = 0.5]{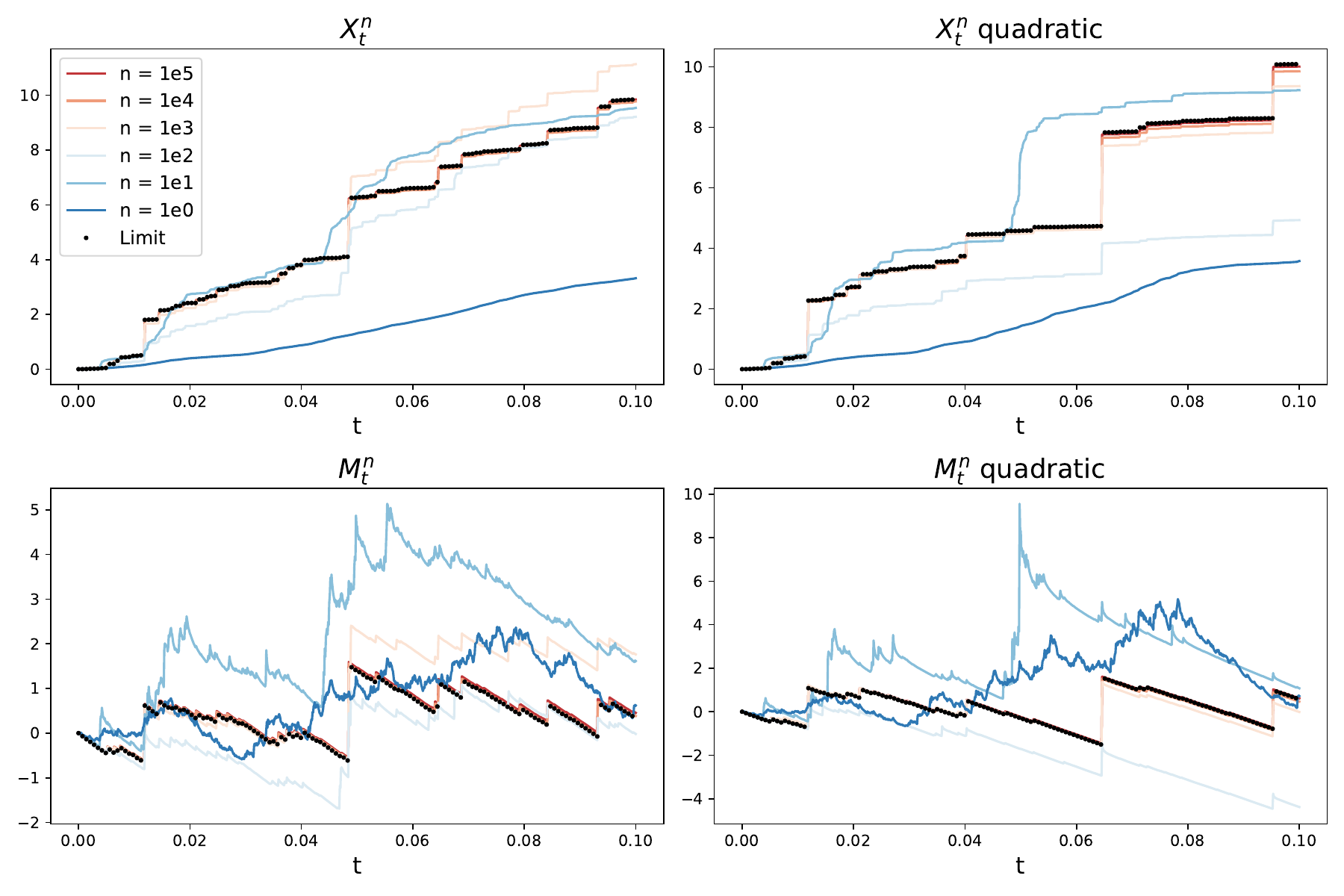}
    
    \caption{Convergence of $X^{n}$ and $M^{n} := W^n_{f(X^{n})}$} from Theorem \ref{thm:weak_cv_mean_rev_process}, with $K(t) = t^{-1/2}\,$, $a \equiv 1\,$, $b(t) = 100 \,e^{-t}\,$, $T = 0.1\,$, and $\lambda = \nu = 1\,$. The black dotted lines correspond to the limiting jump processes simulated with the same random seed. \textbf{Left:} $f(x) = x\,$. \textbf{Right:} $f(x) = x + \frac{1}{2}x^2\,$.
    \label{fig:convergence_fast_mean_rev}
\end{figure}

\subsection{Hyper-rough regime as $\alpha \to 0$}
\label{subsubsec:H_to_-1/2}
In \cite{abi2025hyper}, the integrated rough square-root process
$$
X_t := G_0(t) + \frac{1}{\Gamma(\alpha)}\,\int_0^t\,(t-s)^{\alpha - 1}\,\left(- \lambda \,X_s + \nu\,W_{X_s}\right)\,ds
$$
was shown to converge to an Inverse Gaussian process as $\alpha \downarrow 0\,$. Since the kernel converges weakly to a Dirac measure at zero in this case (see \cite[Lemma~2.4]{abi2025hyper}), this is a direct consequence of Theorem \ref{thm:weak_convergence}. However, our topological framework is more accurate; in particular, it is stronger than the $S$-topology used in \cite{abi2025hyper} (by Proposition \ref{prop:S-conv}) and it offers better control over quantities related to the martingale part $M=W_{X}\,$ in the Volterra equation. Moreover, we can now assert that weak convergence remains true in the non-affine setting, when considering the more general martingale part $M=W_{f(X)}\,$, leading to more general pure-jump limits given by first crossing problems with curved boundaries. This is summarised in the following result.
\begin{thm}
    \label{thm:weak_convergence_hyperrough}
    We consider $\lambda\,, \nu \geq 0\,$, $T > 0\,$, and $f$ satisfying Assumption \ref{assumption:f}. Let $a\,, b : \R_+ \to \R_+\,$, with $a$ non-decreasing and $b$ locally bounded and measurable, and define 
    $$
    G_0^n(t) := \int_0^t\,g_0^n(s)\,ds\,, \quad g_0^n(t) := a(t) + \frac{1}{\Gamma(\alpha_n)}\,\int_0^t\,(t-s)^{\alpha_n - 1}\,b(s)\,ds\,, \quad t \geq 0\,,
    $$
    for a sequence $(\alpha_n)_n \subset (0\,, 1)$ such that $\alpha_n \to 0\,$. For each $n \geq 0\,$, let $(W^n \,, X^n)$ be a weak solution to \eqref{eq:Volterra_clock} on $[0\,, T]$ for the input $(K^n\,, G_0^n\,, f\,, \lambda\,, \nu)\,$, where $K^n(t) := t^{\alpha_n - 1} / \Gamma(\alpha_n)\,$, which exists by Theorem \ref{thm:weak_existence_Volterra_clock_DDS}.
    Then, $(W^n\,, X^n) \Rightarrow (W^*\,, X^*)$ in $D^{\circ}[0\,,T]\,$, where $W^*$ is a standard Brownian motion, and
    $$
    X^*_t := \inf \{s \geq 0\,:\, (1 + \lambda)s - \nu W^*_{f(s)} > \bar a(t) + \bar b(t)\}\,, \quad t \leq T\,,
    $$
    with $\bar a := \int_0^{\cdot}\,a(s)\,ds$ and $\bar b := \int_0^{\cdot}\,b(s)\,ds\,$.
\end{thm}
\begin{proof}
    This is direct consequence of Theorem \ref{thm:weak_convergence}, noting that the kernels converge to a Dirac mass by \cite[Lemma~2.4]{abi2025hyper}.
\end{proof}

\section{Proofs}
\label{sec:proofs}
\subsection{Proof of Theorem \ref{thm:weak_existence_Volterra_clock}}
\label{subsec:proof_existence}
In order to prove Theorem \ref{thm:weak_existence_Volterra_clock}, we first study locally square-integrable kernels, so that we can work directly on the dynamics of $Y$ in \eqref{eq:square-integrable_case}. We then extend our result to $K \in L^1_{loc}(\R_+\,, \R_+)$ for the integrated process $X\,$, by an argument of weak convergence of processes, using the shifted kernel. 

We first have the following existence result.
\begin{lem}
    \label{lemma:existence_Volterra_V}
    Let $\lambda \,, \nu\geq 0\,$, $f$ satisfying Assumption \ref{assumption:f}, $K \in L^2_{loc}(\R_+\,, \R_+)$ a completely monotone kernel, and $g_0 \in \mathcal G_K\,$. Then, there exists a continuous, nonnegative weak solution $V$ to 
    \begin{equation}
    \label{eq:Volterra_V}
    Y_t = g_0(t) + \int_0^t\,K(t-s)\,\left(-\lambda\,Y_s\,ds + \nu\,\sqrt{f'(X_s)\,Y_s}\,dB_s \right)\,, \quad t \leq T\,,
    \end{equation}
    where $X_t := \int_0^t\,Y_s\,ds\,$. Additionally, $\sup_{t \leq T} \,\mathbb E \left[Y_t^p \right]$ is finite for all $p \geq 2\,$.
\end{lem}
\begin{proof}
    For any $x = (x_1\,, x_2)^{\top} \in \R^2\,$, we define 
    $$
    b(x) := (-\lambda\,x_1\,, x_1)^{\top}\,, \quad \text{and}\quad \sigma(x) := (\nu\,\sqrt{(x_1)_+\,f'((x_2)_+)}\,, 0)^{\top}\,,
    $$
    where $(\cdot)_+ := \max(\cdot\,, 0)\,$. Both $b$ and $\sigma$ are continuous, and the linear growth assumption on $f'$ ensures that they have linear growth. Hence, \cite[Theorem~A.1]{abi2019markovian} gives existence of a continuous $\R^2$-valued weak solution $U$ to 
    $$
    U_t = \tilde g_0(t) + \int_0^t\,\tilde K(t-s)\,\left(b(U_s)\,ds + \sigma(U_s)\,dB_s\right)\,,
    $$
    where we have used the notations $\tilde g_0 := (g_0\,, 0)^{\top}\,$, and 
    $$
    \tilde K := \left(\begin{array}{cc} K & 0 \\ 0 & 1\end{array}\right)\,.
    $$
    Writing $U := (Y\,, X)^{\top}\,$, we obtain 
    $$
    \begin{cases}
        Y_t = g_0(t) + \int_0^t\,K(t-s)\,\left(-\lambda\,Y_s\,ds + \nu\,\sqrt{f'((X_s)_+)\,(Y_s)_+}\,dB_s \right) \\
        X_t = \int_0^t\,Y_s\,ds
    \end{cases} \,.
    $$
    Thus, it remains to prove that $Y$ is nonnegative. Using the specific structure of $g_0 \in \mathcal G_{K}\,$, this can be achieved using the same arguments as in the proof of \cite[Theorem~A.2]{abi2019markovian}.

    Finally, the finiteness of the moments of $Y$ can be proved by adapting the argument of \cite[Lemma~3.1]{abi2019affine} with a non-constant $g_0\,$.
\end{proof}

Integrating \eqref{eq:Volterra_V} and applying the stochastic Fubini theorem thanks to the local square-integrability of the kernel $K\,$, we obtain

$$
X_t = G_0(t) + \int_0^t\,K(t-s)\,\left(-\lambda\,X_s + \nu\,M_s\right)\,ds\,, \quad t \leq T\,,
$$
where $G_0 := \int_0^{\cdot}\,g_0(s)\,ds\,$, and $M := \int_0^{\cdot}\,\sqrt{f'(X_s)\,Y_s}\,dW_s$ is a continuous martingale, starting from zero, with quadratic variation $\langle M \rangle = f(X)\,$. Using this Volterra equation, we can obtain more precise estimates on $X\,$, given in the following lemma.

\begin{lem}
    \label{lemma:estimates_bar_V_for_existence}
    Let $(Y\,, X)$ be as in Lemma \ref{lemma:existence_Volterra_V}. Then, there exists a constant $C > 0$ depending on $g_0\,$, $\nu\,$, $T$ and $f$ such that 
    \begin{equation}
        \label{eq:second_moment_bar_V_for_existence}
        \mathbb E \left[\sup_{t \leq T}\,(X_t)^2 \right] \leq C\,\left(1 + \int_0^T\,\left|R_{-C\bar K(T) K}(s) \right|\,ds \right)\,,
    \end{equation}
    where $\bar K := \int_0^{\cdot}\,K(s)\,ds\,$, and $R_{-c\bar K(T) K}$ is the resolvent of the second kind of $t \mapsto -C\bar K(T)\,K(t)$ (see Appendix \ref{app:resolvents}).

    Moreover, for $s \leq t \leq T\,$, we have 
    \begin{equation}
    \label{eq:bound_increments_bar_V_for_existence}
    X_t - X_s \leq G_0(t) - G_0(s) + 2\nu\,\|M\|_{\infty}\, \bar K(t-s)\,.
    \end{equation}
\end{lem}
\begin{proof}
    Since $\lambda\,, K\,, X \geq 0\,$, we have
    $$
    X_t \leq G_0(t) + \nu\,\int_0^t\,K(t-s)\, M_s\,ds\,, \quad t \leq T\,.
    $$
    Applying Jensen's inequality, we get 
    \begin{align*}
    (X_t)^2 &\leq 2\,G_0^2(t) + 2\nu^2\,\bar K(t)\,\int_0^t\,K(t-s)\,M_s^2\,ds \,, \quad t \leq T\,.
    \end{align*}
    Recalling that $M$ is a martingale with quadratic variation $f(X)\,$, and that $f(x) \leq c\,(1 + x^2)$ for some constant $c > 0\,$, taking expectations in the previous gives 
    $$
    \mathbb E \left[(X_t)^2\right] \leq 2\,G_0^2(t) + 2\nu^2\,\bar K(t)\,\int_0^t\,K(t-s)\,\left(c + c\,\mathbb E \left[(X_s)^2 \right] \right)\,ds \,, \quad t \leq T\,.
    $$
    Since $G_0$ is continuous it is bounded on $[0\,,T]\,$, so that using the fact that $K \geq 0\,$, we obtain 
    $$
    \mathbb E \left[(X_t)^2 \right] \leq C + C\,\bar K(T)\,\int_0^t\,K(t-s)\,\mathbb E \left[(X_s)^2\right]\,ds \,, \quad t \leq T\,,
    $$
    for some constant $C > 0\,$, depending on $g_0\,$, $\nu\,$, $T$ and $c\,$. Since $K$ is completely monotone, it is non-decreasing. Hence, the resolvent of the second kind of $-C \bar K(T) K\,$, that we denote by $R_{-C\bar K(T) K}\,$, is nonpositive by virtue of Proposition \ref{prop:nonpositive_resolvent}. We can therefore make use of the Gr\"onwall inequality for linear Volterra equations from \cite[Lemma~9.8.2]{gripenberg1990volterra}, which gives 
    $$
    \mathbb E \left[(X_t)^2\right] \leq C\,\left(1 + \int_0^T\,|R_{-C\bar K(T) K}(s)|\,ds\right)\,, \quad t \leq T\,.
    $$
    Since $X$ is non-decreasing, its supremum over $[0\,,T]$ is $X_T\,$, which proves \eqref{eq:second_moment_bar_V_for_existence}.

    We now move to the second inequality. For $s \leq t \leq T\,$, we decompose the increment of $X$ into 
    \begin{align*}
        X_t - X_s &= G_0(t) - G_0(s) - \lambda\,\int_0^t\,K(r)\,X_{t-r}\,dr + \lambda\,\int_0^s\,K(r)\,X_{s-r}\,dr \\
        &\quad + \nu\,\int_0^t\,K(t-r)\,M_r\,dr - \nu\,\int_0^s\,K(s-r)\,M_r\,dr \\
        &\leq G_0(t) - G_0(s) - \lambda\,\int_0^s\,K(r)\,\left(X_{t-r}- X_{s-r}\right)\,dr \\
        &\quad + \nu\,\int_0^s\,(K(t-r) - K(s-r))\,M_r\,dr + \nu\,\int_s^t\,K(t-r)\,M_r\,dr\\
        &\leq G_0(t) - G_0(s) + \nu\,\int_0^s\,(K(t-r) - K(s-r))\,M_r\,dr + \nu\,\int_s^t\,K(t-r)\,M_r\,dr\,, 
    \end{align*}
        since $\lambda\,, K \geq 0$ and $X$ is non-decreasing. We deduce that
    \begin{align*}
        X_t - X_s &\leq G_0(t) - G_0(s) + \nu\,\|M\|_{\infty}\,\left(\int_0^s\,|K(t-r) - K(s-r)|\,dr + \int_s^t\,K(t-r)\,dr \right) \\
        &= G_0(t) - G_0(s) + \nu\,\|M\|_{\infty}\,\left(\bar K(s) - \bar K(t) + 2 \,\bar K(t-s) \right) \\
        &\leq G_0(t) - G_0(s) + 2\nu\,\|M\|_{\infty} \,\bar K(t-s)\,,
    \end{align*}
    using the fact that $K$ is non-increasing and nonnegative.
\end{proof}

We are now ready to prove Theorem \ref{thm:weak_existence_Volterra_clock}.
\begin{proof}[Proof of Theorem \ref{thm:weak_existence_Volterra_clock}]
    Let us fix $T > 0\,$, $\lambda\,, \nu \geq 0\,$,  $f$ satisfying Assumption \ref{assumption:f}, $K \in L^1_{loc}(\R_+\,, \R)$ completely monotone, and $g_0 \in \mathcal G_K\,$. Then, for any $n \geq 1\,$, $K^n := K(\cdot + 1/n)$ is in $L^2_{loc}(\R_+\,,\R_+)$ and completely monotone. Moreover, since $g_0 \in \mathcal G_K\,$, we can write it as 
    $$
    g_0(t) = a(t) + \int_0^t\,K(t-s)\,b(s)\,ds\,, \quad t \geq 0\,,
    $$
    with $a\,, b \geq 0\,$, $a$ non-decreasing and $b$ locally bounded. Defining
    $$
    g_0^n(t) := a(t + 1/n) + \int_0^t\,K^n(t-s)\,b(s)\,ds\,, \quad t \geq 0\,, \quad n \geq 1\,,
    $$
    we know from \cite[Remark~2.14]{abi2021weak} that $g_0^n \in \mathcal G_{K^n}$ and $\lim_{n \to \infty}\,\int_0^t\,g_0^n(s)\,ds = \int_0^t\,g_0(s)\,ds\,$.
    
    We can thus apply Lemma \ref{lemma:existence_Volterra_V} for each $n \geq 1\,$, giving the existence of a filtered probability space $\left( \Omega^n\,, \mathcal F^n\,, \left(\mathcal F^n_t\right)_{t \leq T}\,, \mathbb P^n\right)$ satisfying the usual conditions, admitting a nonnegative, continuous process $Y^n$ and a standard Brownian motion $B^n\,$, such that 
    $$
    Y^n_t = g_0^n(t) + \int_0^t\,K^n(t-s)\,\left(-\lambda\,Y^n_s\,ds + \nu\,\sqrt{f'(X^n_s)\,Y^n_s}\,dB^n_s\right)\,, \quad t \leq T\,,
    $$
    with $X^n_t := \int_0^t\,Y^n_s\,ds\,$. As previously, integrating this relation gives 
    $$
    X^n_t = G_0^n(t) + \int_0^t\,K^n(t-s)\,\left(-\lambda\, X^n_s + \nu\,M^n_s\right)\,ds \,, \quad t \leq T\,,
    $$
    where $G_0^n = \int_0^{\cdot}\,g_0^n(s)\,ds\,$, and $M^n := \int_0^{\cdot}\,\sqrt{f'(X^n_s)\,Y^n_s}\,dB^n_s$ is a continuous martingale, starting from zero, with quadratic variation $\langle M^n \rangle = f(X^n)\,$. 

    We aim to prove that $(X^n, M^n)_{n \geq 1}$ converges weakly, along a subsequence, to a solution of \eqref{eq:dynamic_bar_V}. The proof is based on tightness of the sequence, before identifying the weak accumulation points.

    \paragraph{Tightness.}We first prove that $\left(X^n\right)_{n \geq 1}$ is C-tight (see \cite[Proposition~VI.3.26]{jacod2013limit}). 
    
    Since $K \in L^1_{loc}(\R_+\,, \R)\,$, $K^n = K(\cdot + 1/n)$ converges to $K$ in $L^1_{loc}(\R_+\,, \R)\,$. Hence $\bar K^n$ converges pointwise to $\bar K\,$. From \eqref{eq:second_moment_bar_V_for_existence}, there exists $C > 0$ such that 
    \begin{equation}
\label{eq:sup_n_finite_second_moment_bar_V_for_existence}
    \sup_{n \geq 1}\,  \mathbb E \left[\sup_{t \leq T}\,\left(X^n_t\right)^2 \right] \leq \sup_{n \geq 1}\, C\, \left(1 + \int_0^T\, |R_{-C \bar K^n(T) K^n}(s)|\,ds \right)\, < + \infty\,,
    \end{equation}
    using the $L^1$ continuity of the resolvent of the second kind (see \cite[Theorem~2.3.1]{gripenberg1990volterra}). This implies that $\sup\limits_{n \geq 1}\, \mathbb P^n\left(\sup\limits_{t \leq T} X^n_t \geq R \right) \overset{R \to + \infty}{\longrightarrow} 0\,$. Moreover, using \eqref{eq:bound_increments_bar_V_for_existence}, we have 
    $$
    \mathbb E \left[w(X^n, \delta) \right] \leq w(G_0, \delta) + 2\nu\,\mathbb E \left[\| M^n\|_{\infty} \right]\, \bar K^n(\delta)\,, \quad \delta > 0\,,
    $$
    where $w(\cdot\,, \cdot)$ is the modulus of continuity.
    We can then use Doob's maximal inequality for $\mathbb E \left[\|M^n\|^2_{\infty} \right] \leq 4\,\mathbb E \left[f(X^n_T)\right]\,$, so that by the quadratic growth assumption of $f$ and \eqref{eq:sup_n_finite_second_moment_bar_V_for_existence}, we obtain $\sup_{n \geq 1}\,\mathbb E\left[\|M^n\|_{\infty}^2\right] < + \infty\,$. Moreover, $w(G_0, \delta) \overset{\delta \to 0^+}{\longrightarrow} 0$ by continuity, and $\limsup_{n \to + \infty} \bar K^n(\delta) = \bar K(\delta) \overset{\delta \to 0^+}{\longrightarrow} 0\,$. Hence $\lim\limits_{\delta \to 0^+}\,\limsup\limits_{n \to \infty}\,\mathbb E \left[w(X^n, \delta) \right] = 0\,$, proving that $\left(X^n\right)_{n \geq 1}$ is C-tight.

    We now prove that $\left(M^n\right)_{n \geq 1}$ is C-tight. Since it is continuous, we only have to prove it is tight \cite[Proposition~VI.3.26]{jacod2013limit}. For this, we have to show that the sequence of its quadratic variations $\left(f(X^n)\right)_{n \geq 1}$ is C-tight \cite[Theorem~IV.4.13]{jacod2013limit}. We already have 
    $$
    \sup_{n \geq 1}\,\mathbb E \left[\sup_{t \leq T}\, f(X^n_t) \right] < + \infty
    $$
    in virtue of \eqref{eq:sup_n_finite_second_moment_bar_V_for_existence}, since $f$ has quadratic growth.
    More precisely, there exists $c_1 > 0$ such that $f'(x) \leq c_1\,(1 + x)\,$, so that we can bound 
    $$
    f(X^n_t) - f(X^n_s) \leq c_1\,\left(X^n_t - X^n_s + \frac{(X^n_t)^2}{2} - \frac{(X^n_s)^2}{2}\right) \leq c_1\,(1 + X^n_T) \,w(X^n, \delta)
    $$
    for any $s \leq t \leq T$ such that $t -s \leq \delta\,$, since $X^n$ is non-decreasing.
    Thus, we can prove that $\lim\limits_{\delta \to 0^+}\,\limsup\limits_{n \geq 1}\,\mathbb E \left[w(f(X^n), \delta) \right] = 0$ with similar arguments than for $\left(X^n\right)_{n \geq 1}$, using the fact that 
    $$
    \sup_{n \geq 1}\, \mathbb E \left[X^n_T\,\left\|M^n\right\|_{\infty} \right] \leq \sup_{n \geq 1}\, \sqrt{\mathbb E \left[\left(X^n_T\right)^2 \right]\, \mathbb E \left[\|M^n\|_{\infty}^2 \right]} < + \infty\,.
    $$
    Hence, the sequence $\left(X^n\,, M^n\,, f(X^n)\right)_{n \geq 1}$ is C-tight.

    \paragraph{Identifying the limit.} By the Prokhorov theorem, tightness implies weak convergence along a subsequence (that we do not relabel here) $\left(X^n\,, M^n\,, f(X^n) \right) \overset{n \to \infty}{\implies} (X\,, M\,, f(X))\,$. These limits are almost surely continuous processes. We have already seen that $\sup\limits_{n \geq 1}\, \sup\limits_{t \leq T} \mathbb E \left[\left(M^n_t\right)^2\right] < + \infty$ thanks to Doob's maximal inequality, thus $\left(M^n_t\right)_{t \leq T, n \geq 1}$ is uniformly integrable and $M$ is martingale with respect to the filtration generated by $(X\,, M)$ in virtue of \cite[Proposition IX.1.12]{jacod2013limit}. Moreover, \cite[Corollary~VI.6.29]{jacod2013limit} shows that $\langle M \rangle = f(X)\,$.

    Finally, applying Skorokhod's representation theorem, we have the existence of a probability space $\left(\Omega\,, \mathcal F\,, \mathbb P\right)$ admitting copies of $\left(X^n, M^n\,, f(X^n) \right)_{n \geq 1}$ and $\left(X\,, M\,, f(X)\right)$ (again, not relabeled) such that the convergence holds almost surely in the uniform topology (we recall that the Skorokhod topology coincides with the uniform topology for continuous limits). For all $t \leq T\,$, we have
    \begin{align*}
    \int_0^t\,K^n(t-s)\,Z^n_s\,ds - \int_0^t\,K(t-s)\,Z_s\,ds &= \int_0^t\,K(t-s)\,\left(Z^n_s - Z_s\right)\,ds \\
    &\quad+ \int_0^t \left(K^n(t-s) - K(t-s)\right)\,Z^n_s\,ds \,, \quad n \geq 1\,,
    \end{align*}
    with $Z^n := - \lambda X^n + \nu\,M^n$ and $Z := -\lambda\,X + \nu\,M$\,, which converges to zero in the limit $n \to + \infty$ thanks to the dominated convergence theorem. Also, we know that $(G_0^n)_{n \geq 1}$ converges pointwise to $G_0 := \int_0^{\cdot}\,g_0(s)\,ds\,$. By continuity of the limiting processes, 
    $$
    X_t = G_0(t) + \int_0^t\,K(t-s)\,\left(-\lambda \,X_s + \nu\,M_s\right)\,ds \,, \quad t \leq T\,, \quad \mathbb P-a.s.\,,
    $$
    which concludes the proof of weak existence.
\end{proof}

\subsection{Proof of Theorem \ref{thm:weak_convergence}}
\label{subsec:proof_convergence}
Before proving the weak convergence of processes of Theorem \ref{thm:weak_convergence}, we begin with studying the limiting behavior of deterministic equations similar to \eqref{eq:Volterra_clock}. We recall that the space $D^{\circ}[0\,,T]$ is defined in Section \ref{sect:func_conv_frame}. In the following lemma, $G_0^n\,$, $K^n\,$, $f^n$ and $f$ are as in Theorem \ref{thm:weak_convergence}.

\begin{lem}
    \label{lemma:limit_deterministic_equation}
    Let $x^n := g^n \circ f^n \circ \tau^n\,$, and $x := g \circ f \circ \tau\,$, where $(g^n\,, \tau^n) \to (g \,, \tau)$ in $D^{\circ}[0\,,T]\,$. Suppose also that $\tau(T) = \tau(T-)\,$. Let $\tau^n$ and $x^n$ be such that $\tau^n(0) = x^n(0) = 0\,$, and 
    $$
    \tau^n(t) = G_0^n(t) + \int_0^t\,K^n(t-s)\,(-\lambda^n\, \tau^n(s) + \nu^n\,x^n(s))\,ds\,, \quad t \leq T\,,
    $$
    with $\lambda^n \to \lambda$ and $\nu^n \to \nu\,$.
    Then we have 
    $$
    (1 + \lambda)\,\tau(t) = G_0(t)  + \nu\,x(t)\,, \quad t \leq T\,.
    $$
\end{lem}
\begin{proof}
    Firstly, Proposition \ref{prop:conv_compositions} gives $\tau^n(0) \to \tau(0)$ and $x^n(0) \to x(0)\,$, so $\tau(0) = x(0) = G_0(0) = 0\,$, and hence the result trivially holds at $t = 0\,$. Next, we recall that $G_0^n(t) \to G_0(t)$ for all $t \in [0\,,T]\,$. Turning to the integral, we set
    $$
    y^n := - \lambda^n\, \tau^n + \nu^n\,x^n \quad \text{and} \quad y := - \lambda \,\tau + \nu\,x\,.
    $$
    We recall that for non-decreasing functions, convergence in the $M_1$ topology is equivalent to pointwise convergence on a dense subset of $[0\,,T]\,$, including $0$ and $T$ (see \cite[Corollary~12.5.1]{whitt2002stochastic}).
    We thus still have $f^n \circ \tau^n \to f \circ \tau$ in the $M_1$ topology, by the locally uniform convergence of $f^n$ to $f\,$. Hence $\sup_{n}\,\sup_{s \leq T}\,|x^n(s)| < + \infty$ from Proposition \ref{prop:conv_compositions}. Moreover, $\sup_n \sup_{s \leq T}\,|\tau^n(s)| = \sup_n \tau^n(T) < + \infty\,$, since $\tau^n(T) \to \tau(T)\,$. This leads to
    $$
    \sup_n \,\sup_{s \leq T}\,|y^n(s)| < + \infty\,.
    $$
    For any given $\delta \in (0\,, t)\,$, we then have 
    \begin{align}
    \label{eq:bound_cv_integral_deterministic}
        \left|\int_0^t\,K^n(s)\,y^n(t-s)\,ds - y(t) \right| & \leq \int_0^{\delta}\,K^n(s)\,|y^n(t-s) - y(t)|\,ds \nonumber\\
        &\quad + |y(t)|\,\left|\int_0^{\delta}\,K^n(s)\,ds - 1\right| + \int_{\delta}^t\,K^n(s)\,|y^n(t-s)|\,ds\,.
    \end{align}
    The second term vanishes in the limit thanks to the convergence $(K^n)_n\,$, and the last one is bounded by 
    $$
    \sup_{s \leq t}\,|y^n(s)|\,\int_{\delta}^t\,K^n(s)\,ds \overset{n \to \infty}{\longrightarrow} 0\,,
    $$
    thanks to the uniform bound we derived on $(y^n)_n\,$, and the convergence of $(K^n)_n\,$.

    Now, fix an arbitrary $t \notin D(\tau)\,.$ It follows from Proposition \ref{prop:conv_compositions} that, for any $\varepsilon > 0\,$, we can find a small enough $\delta \in (0\,, t)$ so that 
    $$
    \limsup_{n \to \infty}\, \sup\{|y^n(s) - y(t)|\,:\,s \in[t - \delta\,, t] \} < \varepsilon\,,
    $$
    and hence
    $$
    \int_0^{\delta}\,K^n(s)\,|y^n(t-s) - y(t)|\,ds \leq \varepsilon\,\int_0^{\delta}\,K^n(s)\,ds \leq \varepsilon\,,$$
    for large enough $n \geq 0\,$. Since $\varepsilon > 0$ was arbitrary, this means that the first term in the right-hand side of  \eqref{eq:bound_cv_integral_deterministic} also vanishes. Thus we have shown that 
    $$
    G_0^n(t) + \int_0^t\,K^n(t-s)\,\left(- \lambda^n\, \tau^n(s) + \nu^n\,x^n(s)\right)\,ds \to G_0(t) - \lambda\, \tau(t) + \nu\,x(t)
    $$
    for all $t \notin D(\tau)\,$. By assumption, this includes $t = T\,$. Since we also have $\tau^n(t) \to \tau$ for $t \in \{0\,,T\}$ and $t \notin D(\tau)\,$, we conclude that the result holds for any $t \in [0\,,T]\,$, by right-continuity of all the limiting quantities.
\end{proof}

As expected, the convergence of the sequence of kernels to a Dirac measure leads to a pointwise equation in the limit, which no longer has a Volterra structure. 

We now aim to prove tightness of our processes. Since $X^n$ is non-decreasing, and starting from zero for all $n \geq 0\,$, the tightness criterion in the $M_1$ topology from \cite[Theorem~12.12.3]{whitt2002stochastic} reduces to 
$$
\lim_{R \to + \infty}\,\limsup_{n \to + \infty}\,\mathbb P(X^n_T > R) = 0 \, \quad \text{and} \quad \lim_{\delta \to 0^+}\,\limsup_{n \to + \infty}\,\mathbb P(X^n_{\delta} \vee (X^n_T - X^n_{T-\delta}) > \eta) = 0\,, \quad \eta > 0\,.
$$

\begin{lem}[$M_1$ tightness]
    \label{lem:M1_tightness_time_changed}
    We have
    \begin{equation}
        \label{eq:expectation_increment_bar_V}
        \mathbb E \left[X^n_t - X^n_s\right] \leq G_0^n(t) - G_0^n(s)\,, \quad s \leq t \leq T\,, \quad n \geq 0\,.
    \end{equation}
    In particular, since $G_0^n$ converges pointwise to $G_0$ which is continuous in $0$ and $T\,$,
    $$
    \sup_{n \geq 1}\, \mathbb E \left[X^n_T\right] < + \infty\,, \quad \text{and}\quad \lim_{\delta \to 0^+}\, \limsup_{n \to \infty}\, \mathbb E \left[X^n_{\delta}\right] = \lim_{\delta \to 0^+}\, \limsup_{n \to \infty}\, \mathbb E \left[X^n_{T} - X^n_{T- \delta}\right] = 0\,,
    $$
    and $\left(X^n\right)_{n \geq 1}$ is tight on $D[0\,,T]$ for the $M_1$-topology.
\end{lem}
\begin{proof}
    We have 
    \begin{align*}
        X^n_t - X^n_s &= G_0^n(t) - G_0^n(s) + \int_0^s\,K^n(r)\,\left(\lambda\,X^n_{s-r} - \lambda\,X^n_{t-r} + \nu\,M^n_{t-r} - \nu\,M^n_{s-r}\right)\,dr \\
        &\quad + \int_s^t\,K^n(r)\,\left(- \lambda\,X^n_{t-r} + \nu\,M^n_{t-r}\right)\,dr \,, \quad s \leq t \leq T\,, \quad n \geq 0\,,
    \end{align*}
    so that using $\lambda\,, K^n \geq 0$ and the non-decreasing property $X^n\,$, we obtain
    \begin{align*}
    X^n_t - X^n_s &\leq G_0^n(t) - G_0^n(s) + \nu\,\int_0^s\,K^n(r)\,\left(M^n_{t-r} - M^n_{s-r}\right)\,dr \\
    &\quad+ \nu\,\int_s^t\,K^n(r)\,M^n_{t-r}\,dr\,, \quad s \leq t \leq T\,, \quad n \geq 0\,.
    \end{align*}
    Taking expectations on both sides proves the result, since $M^n$ is a martingale starting from zero.
\end{proof}

The identification of the limiting process in Theorem \ref{thm:weak_convergence} relies on proving that the limit of $W^n \circ f^n(X^n)$ is a martingale. More precisely, Lemma \ref{lemma:limit_deterministic_equation} allows us to identify the limiting equation, but martingality is required to prove that the limiting process is the first time satisfying such equation. We will thus use the following proposition.

\begin{prop}
\label{prop:martingality_limit}
	Let $Z^n$ be continuous martingales and let each $A^n_t$ be a stopping time for $t\geq0$ and $n\geq 1\,$, such that $A^n_0 = 0\,$. If $(Z^n,A^n) \Rightarrow (Z,A)$ in $D^\circ[0,T]$ and $M^n=Z^n\circ A^n$ satisfies $\sup_{n
	\geq 1}\mathbb{E}[|\sup_{t\in[0,T]}M^n_t|]<\infty$ then $M=Z\circ A$ is a martingale, and 
    $$
    \mathbb E \left[\sup_{t \leq A_T}\,|Z_t| \right] < + \infty\,.
    $$
\end{prop}
\begin{proof}
	By applying Skorokhod's representation theorem, let the convergence hold almost surely. Define
	\[
	\tau^n_c := \inf \{ t\geq 0 : |Z^n_t|\geq c\},\quad 	\tau_c := \inf \{ t\geq 0 : |Z_t|\geq c\},
	\]
	and note that we must have
	$\mathbb{P}( \tau_c< \lim_{\varepsilon\downarrow 0} \tau_{c+\varepsilon})=0$ for all but (at most) countably many $c>0$, since $c\mapsto \tau_c$ is a left-continuous (non-decreasing) process. Fix a sequence $c_k\uparrow \infty$ for which $\mathbb{P}( \tau_{c_k}= \lim_{\varepsilon\downarrow 0} \tau_{c_k+\varepsilon})=1$. Then we have $\tau^n_{c_k} \rightarrow \tau_{c_k}$ almost surely as $n\rightarrow \infty$, for each $k\geq 1$. As in Proposition \ref{prop:conv_compositions}, we get $Z^n_{A^n_t\land \tau^n_{c_k}}\rightarrow
Z_{A_t\land \tau_{c_k}}$ almost surely, as $n\rightarrow \infty$, for a co-countable set of times $t\in \mathbb{T}\subseteq[0,T]$. Since $Z^n$ stopped at $\tau^n_{c_k}$ is bounded, and $A^n_t$ is a stopping time, we have
\begin{equation}\label{eq:X^n_martingale}
\mathbb{E}\Bigl[ \bigl(Z^n_{A^n_t\land \tau^n_{c_k}} - Z^n_{A^n_s\land \tau^n_{c_k}}\bigr)\prod_{i=1}^nf_i(Z^n_{A^n_{s_i}\land \tau^n_{c_k}})\Bigr] = 0.
\end{equation}
for any bounded continuous functions $f_i$ and any $s_i\leq s\leq t$. Let $s_i,s,t \in \mathbb{T}$. Since the $(Z^n_{A^n_t\land \tau^n_{c_k}})_{t\in[0,T]}$ are bounded by $c_k$ uniformly in $n\geq 0$, it follows from \eqref{eq:X^n_martingale} and the convergence above that
\begin{equation}\label{eq:X_martingale}
\mathbb{E}\Bigl[ \bigl(Z_{A_t\land \tau_{c_k}} - Z_{A_s\land \tau_{c_k}}\bigr)\prod_{i=1}^nf_i(Z_{A_{s_i}\land \tau_{c_k}})\Bigr] = 0.
\end{equation}
Now, applying Corollary \ref{cor:inf_sup} and Fatou's lemma, we get
\[
\mathbb{E}[|\sup_{t\in[0,A_T]}Z_t|]\leq \liminf_{n\rightarrow \infty}\mathbb{E}[|\sup_{t\in[0,T]}M^n_t|] <\infty\,,
\]
which proves the last statement of the proposition.
Sending $k\rightarrow \infty$ in \eqref{eq:X^n_martingale}, we can therefore apply dominated convergence to conclude that
\begin{equation}\label{eq:M_martingale}
	\mathbb{E}\Bigl[ (M_t - M_s)\prod_{i=1}^nf_i(M_{s_i})\Bigr] = 0,
\end{equation}
for all bounded continuous $f_i$ and all $s_i \leq s \leq t$ in $\mathbb{T}$. Finally, $|\sup_{t\in[0,T]}M_t|\leq |\sup_{t\in[0,A_T]}Z_t|$, so, by right-continuity of $M$ and dominated convergence, we obtain the martingale property of $M$ from \eqref{eq:M_martingale}. This completes the proof.
\end{proof}

We are now ready to prove Theorem \ref{thm:weak_convergence}.

\begin{proof}[Proof of Theorem \ref{thm:weak_convergence}]
    By Lemma \ref{lem:M1_tightness_time_changed}, the pairs $(W^n\,, X^n)$ are tight on $D^{\circ}[0\,,T]\,$. Let $(W^*\,, X^*)$ denote an arbitrary realisation of the limit along a weakly convergent subsequence $(W^n\,, X^n)$ which we fix throughout and still index by $n\,$. Trivially, $W^*$ is a standard Brownian motion. Define
     $$
     \mathbb I := \{t \in [0\,,T]\,:\, \mathbb P(X_t^* = X_{t-}^*) = 1 \}\,,
     $$
     which is a co-countable subset of $[0\,,T]\,$. For the $M_1$ topology on $D[0\,,T]\,$, the marginal projections $\tau \mapsto (\tau(t_1)\,, \ldots\,, \tau(t_m))$ are continuous at $t_i \in \mathbb I \cup \{T\}$ (see Proposition \ref{prop:conv_compositions} with $f^n = Id$). Using Lemma \ref{lem:M1_tightness_time_changed}, for any $\varepsilon > 0\,$, we can find a sequence $\delta_k \downarrow 0$ with $T - \delta_k \in \mathbb I$ so that 
     $$
     \mathbb P(X^*_T - X^*_{T - \delta_k} > \varepsilon) \leq \liminf_{n \to \infty}\,\mathbb P (X^n_T - X^n_{T - \delta_k} \geq \varepsilon) \leq \eta_k\,,
     $$
     with $\eta_k \downarrow 0$ as $k \to \infty\,$. Consequently, $X^*_{T - \delta_k} \to X^*_T$ in probability and hence $X^*_T = X^*_{T-}$ almost surely, so $T \in \mathbb I\,$. Using this, we can apply Skorokhod's representation theorem on the Polish space $D^{\circ}[0\,,T]$ and invoke Lemma \ref{lemma:limit_deterministic_equation} to see that, with probability 1,
     \begin{equation}
     \label{eq:limit_eq_subseq}
     (1 + \lambda)\,X^*_t = G_0(t) + \nu\,M^*_t\,, \quad t \leq T\,,
     \end{equation}
      where $M^*_t := W^*_{f(X^*_t)}$ for a suitable realisation of the weak limit point $(W^*\,, X^*)$ which we fix from here on. It remains to prove that for each $t \in [0\,,T]\,$, $X^*_t$ is the first time satisfying \eqref{eq:limit_eq_subseq}.
     By the BDG inequality, the uniform quadratic growth assumption on $(f^n)_{n \geq 0}\,$, and Lemma \ref{lem:M1_tightness_time_changed},
     we have $\sup_{n \geq 0}\,\mathbb E \left[ \sup_{t \leq T}\,|W^n_{f^n(X^n_t)}|\right] < + \infty\,$, so we conclude that $M^*$ is a martingale by virtue of Proposition \ref{prop:martingality_limit}. Moreover, from \eqref{eq:limit_eq_subseq} we deduce that $\mathbb E \left[X^*_t \right]$ is finite for all $t \leq T\,$.
     Next, we define
     $$
     U^*_t := \inf\{s \geq 0\,:\,(1 + \lambda)s - \nu\,W^*_{f(s)} > G_0(t)\}\,, \quad t \leq T\,,
     $$
     which is almost surely finite, as $G_0 \geq 0$, $1+\lambda >0$, and $f(0) = 0\,$. By continuity of $G_0$, $U^*$ is c\`adl\`ag. Moreover, using again that $1+\lambda >0$, and that $f$ is continuous and non-decreasing, it follows from standard properties of Brownian motion that we have
     \[
      U^*_t = \inf\{s \geq 0\,:\,(1 + \lambda)s - \nu\,W^*_{f(s)} \geq G_0(t)\}
     \]
     almost surely, at every $t\leq T$. Since \eqref{eq:limit_eq_subseq} holds for all $t\leq T$ with probability 1, we get $U^*_t \leq X^*_t$ almost surely, for each $t \leq T$. By right-continuity of $U^*$, we can thus conclude that $U^*_t \leq X^*_t$  for all $t\leq T$ with probability 1. In particular, we thus have 
     $$
     \mathbb E \left[\sup_{t \leq f(U^*_T)}\,|W^*_t| \right] \leq \mathbb E \left[\sup_{t \leq f(X^*_T)}\,|W^*_t| \right] < + \infty\,,
     $$
     thanks to Proposition \ref{prop:martingality_limit} as above. Since $f(U^*_t)$ is a stopping time with respect to the filtration generated by $W^*$ for any $t \in [0\,,T]\,$, we can apply Doob's optional sampling theorem to obtain $\mathbb E \left[W^*_{f(U^*_t)}\right] = 0\,$.
     Moreover, we of course have, from the definition of $U^*\,$, that 
     $$
     (1 + \lambda)\,U^*_t = G_0(t) + \nu\,W^*_{f(U^*_t)}\,, \quad t \leq T\,,
     $$
     so that finally
     $$
     \mathbb E \left[X^*_t  - U^*_t \right] = \frac{\nu}{1 + \lambda}\,\mathbb E \left[M^*_t - W^*_{f(U^*_t)} \right] = 0\,, \quad t \leq T\,.
     $$
     Since we also have $X^*_t \geq U^*_t$ for all $t\leq T$ with probability 1, we deduce from this that $X^*_t = U^*_t\,$, for all $t \leq T$, with probability 1, by the right-continuity of $X^*$ and $U^*$. The law of $(W^*\,, X^*)$ is thus uniquely identified by $(W^*,U^*)$, regardless of the weakly convergent subsequence, which completes the proof.
\end{proof}

\subsection{Proof of Theorem \ref{thm:bursts}}\label{subsect:rates_proof}
	
	\begin{proof}[Proof of Theorem \ref{thm:bursts}] As in the proof of Theorem \ref{thm:weak_convergence}, we have tightness of $(F^n,X^n)$ on $D^\circ[0,T]$ and the weak limit points are characterised uniquely by the law of $(F^*,X^*)$ for the limit $X^*$ given by Theorem \ref{thm:weak_convergence}. From there, it follows from Corollary \ref{prop:cont_image_jump} and Proposition \ref{prop:pointwise_char_decorated} that we have weak convergence of $\Upsilon^n=F^n \circ X^n$ to $(F^*\circ X^*, I^*)$ on $\mathfrak{D}[0,T]$ with $I_t^*=F^*([X^*(t-),X^*(t)])$ for $t\in[0,T]$. It remains to observe that we have
		\[
		(F^*\circ X^*)_t = \nu W^*_{f(X^*_t)} - \lambda X^*_t = X^*_t - G_0(t)
		\]
		for $t\in[0,T]$, where the second equality follows as in Theorem \ref{thm:weak_convergence}. This confirms that the base path is indeed $F^*\circ X^*=X^*-G_0$, and it likewise confirms that we simply have $I_t^*=\{X^*_t-G_0(t)\}$ for $t\notin D(X^*)$.
	\end{proof}

\subsection{Proof of the results in Section \ref{sect:func_conv_frame}}
\label{subsect:func_conv_frame_proofs}

Here we give the proofs of the statements that comprise our framework for functional convergence in Section \ref{sect:func_conv_frame}.

\begin{proof}[Proof of Proposition \ref{prop:conv_compositions}] Firstly, by the $M_1$-convergence of $\tau^n$, there is a $K>0$ so that $\tau^n(t)\leq K$ for all $t\in[0,T]$ and all $n\geq1$. Thus, the uniform convergence of $f^n$ to $f$ on $[0,K]$ gives that $\sup_{t\in[0,T]} |x^n(t)| $ is bounded in $n\geq 1$, as desired. Now fix any $\varepsilon>0$. Since the $f^n$ are uniformly convergent, there is a $\delta>0$ so that
	\[
	\limsup_{n\rightarrow \infty}  \,\sup \{ |f^n(s)-f(r)|: s,r\geq0,\, |s-r| \leq \eta \} < \epsilon.
	\]
	By \cite[Theorem 12.5.1(v)]{whitt2002stochastic}, for any $t\notin D(\tau)$, we can then find an $\eta>0$ so that
	\[
	\,\sup \{ |\tau^n(s)-\tau(r)|: s,r \in[0,T]\cap[t-\delta,t+\delta]\} < \eta.
	\]
	for all large enough $n \geq 1$. This establishes \eqref{eq:pointwise}. Moreover, $\tau^n(T)\rightarrow \tau(T)$, so we also have $x^n(T)\rightarrow x(T)$.
\end{proof}

\begin{proof}[Proof of Proposition \ref{prop:S-conv}]
	Since the $\tau^n$ are relatively compact for the $M_1$ topology, $\sup_{t\in[0,T]}|\tau^n(t)|\leq K$ for all $n\geq 1$,
	for some $K>0$. Fix any $\eta>0$ and let $N_\eta^K(f^n)$ be the number of $\eta$-oscillations of $f^n$ on $[0,K]$, i.e., $N_\eta^K(f^n)\geq k$ if and only if there exist $t_1 \leq t_2 \leq \cdots \leq t_{2k}$ in $[0,K]$ so that $|f^n(t_{2i-1})-f^n(t_{2i})|>\eta$ for $i=1,2,...,k$. By the $C$-relative compactness of the $f^n$, there is a small enough $\delta = \delta(\eta)$ such that
	\[
	\sup\{|f^n(t)-f^n(s)| :s,t\in[0,K],\, |s-t| \leq  \delta \} < \eta
	\]
	for all $n\geq 1$. Since $x^n=f^n\circ \tau^n$, it follows that $N_\eta^K(x^n)\geq k$ requires $|\tau^n(t_{2i-1})-\tau^n(t_{2i-1})| > \delta$ for $t_1 \leq t_2 \leq \cdots \leq t_{2k}$ in $[0,T]$. Thus,
	\[
	N_\eta^K(x^n) \leq N_\delta^T(\tau^n)
	\]
	for all $n\geq 1$. However, the $M_1$ relative compactness of the $\tau^n$ also implies that the latter quantity is bounded in $n\geq 1$ (see \cite[Corollary A.9]{sojmark2023weak}). Recall also from Proposition \ref{prop:conv_compositions} that $\sup_{t\in[0,T]}|x^n(s)|$ is bounded. Consequently, the conclusion follows from \cite[Theorem 2.13(iv)]{jakubowski1997non} and \cite[Proposition 2.14]{jakubowski1997non}.
\end{proof}

\begin{proof}[Proof of Proposition \ref{prop:cont_image_jump}]
	Fix $t\in D(\tau)$. Write $a_{n,\delta}=\tau^n(t-\delta\lor 0)$,
	$b_{n,\delta}=\tau^n(t+\delta \land T)$, $a= \tau(t-)$ and $b=\tau(t)$.
	Since $f^n$ and $f$ are continuous,
	$f^n([a_{n,\delta},b_{n,\delta}])=[m_n,M_n]$ and
	$f([\tau(t-),\tau(t)])=[m,M]$ for suitable constants.
	Since $\tau^n$ is non-decreasing,
	$\tau^n(\overline{B(t,\delta)}) \subseteq
	[a_{n,\delta}, b_{n,\delta}]$ and every point in
	$[a_{n,\delta}, b_{n,\delta}]$ is within a distance
	$\sup_{s\leq T}\Delta\tau^n(s)$ of
	$\tau^n(\overline{B(t,\delta)})$, so
	\[
	d_{\mathrm{H}}\bigl(x^n(B(t,\delta)),\,
	[m_n,M_n]\bigr)
	\leq \sup_{|x-y| \leq \sup_{s\leq T}\Delta\tau^n(s)}
	\!\!|f^n(x)-f^n(y)|
	\]
	where the supremum is restricted to $x,y \in [0,K]$ since $\sup_n \tau^n(T)\leq K$, for some $K>0$, and hence the right-hand side tends to zero. By the triangle inequality, it therefore suffices to show
	\begin{equation}\label{eq:hausdorff_interval}
		d_\mathrm{H}\bigl([m_n,M_n],\,
		f([\tau(t-),\tau(t)])\bigr)
		= |m_n -m|\lor |M_n-M|
		\rightarrow 0.
	\end{equation}
	Note that
	$M_n = \max f^n([a_{n,\delta},b_{n,\delta}])$
	and $M=\max f([a,b])$.
	For any given $\eta>0$, by taking $\delta>0$ small enough,
	we can find $l,l^\prime,r,r^\prime\notin D(\tau)$ so that
	$l\leq (t-\delta) \lor 0 \leq l^\prime < t$ and
	$t< r \leq (t+\delta) \land T \leq r^\prime$ with
	\[
	\tau(t-)-\eta \leq \tau(l)\leq \tau(l^\prime) \leq
	\tau(t-),\quad \tau(t) \leq \tau(r)\leq \tau(r^\prime)
	\leq \tau(t) + \eta.
	\]
	Since $\tau^n(l)\leq a_{n,\delta}\leq\tau^n(l^\prime)$
	and $\tau^n(r)\leq b_{n,\delta}\leq\tau^n(r^\prime)$
	with convergence as $n\rightarrow\infty$ for
	$l,l^\prime,r,r^\prime \notin D(\tau)$, it follows that:
	given $\eta>0$, we can take $\delta>0$ small enough so that
	\begin{equation}\label{eq:intervals_close}
		a-2\eta \leq a_{n,\delta}\leq a+\eta,\quad
		b-\eta \leq b_{n,\delta} \leq b+2\eta.
	\end{equation}
	By the triangle inequality, we have
	\[
	|M_n - M|
	\leq \sup_{s\in[0,K]}|f(s)-f^n(s)| + 
	|\max f([a_{n,\delta}, b_{n,\delta}])
	- \max f([a, b])|,
	\]
	and \eqref{eq:intervals_close} gives
	$d_{\mathrm{H}}([a_{n,\delta}, b_{n,\delta}],
	[a, b]) \leq 2\eta$ for all $n$ large enough,
	so the second term on the right-hand side is bounded by the maximum of $|f(x)-f(y)|$ over $|x-y| \leq 2\eta$ with $x,y\in[0,K]$, for all $n$ large enough.
	Hence, by taking $\delta>0$ small enough, we can make
	$\limsup_n |M_n - M| $ as small as we like. The same arguments apply to $|m_n - m|$, so we obtain
	\eqref{eq:hausdorff_interval} and hence the
	proof is complete.
\end{proof}

The proof of Corollary \ref{cor:sum_hausdorff_range} follows in exactly the same way as Proposition \ref{prop:cont_image_jump}.

\begin{proof}[Proof of Corollary \ref{cor:inf_sup}]
	We only consider the infimum. Since $\tau^n(0) \to \tau(0)$ and
	$\phi^n \to \phi$ uniformly on compacts (with $\phi$ continuous), we deduce that $\underline{\phi}^n\to \underline{\phi}$ uniformly on compacts, and hence
	$(\underline{\phi}^n, \tau^n) \to
	(\underline{\phi}, \tau)$ in
	$D^\circ[0,T]$. Next, we can observe that
	\[
	0 \leq \inf_{s\leq t}x^n_s
	- 
	\underline{\phi}^n \circ \tau^n
	\leq \sup_{|x-y| \leq \sup_{s\leq t}\Delta\tau^n(s)}
	\!\!|\phi^n(x)-\phi^n(y)|,
	\]
	where the supremum is restricted to $x,y\leq \sup_n\tau^n(T)<\infty$, so we obtain the uniformly vanishing error $\varepsilon^n(t)$ from the equicontinuity of the $\phi^n$. Finally, since $t\mapsto \inf_{s\leq t} x_s^n$ is non-decreasing, Proposition \ref{prop:conv_compositions} applied to 	$(\underline{\phi}^n, \tau^n) \to
	(\underline{\phi}, \tau)$ gives $M_1$ convergence.
\end{proof}

\begin{proof}[Proof of Corollary \ref{cor:skorokhod}]
	As in the proof of Corollary~\ref{cor:inf_sup}, also
	$\hat{\phi}^n \rightarrow \hat{\phi}$ uniformly on compacts, so the first two claims follow by arguing in the same way. For the final claim, the $M_1$ convergence of $l^n$ follows by Corollary~\ref{cor:inf_sup}, while the second part follows from Corollary \ref{cor:sum_hausdorff_range} and Remark \ref{rem:nearly_cont_phi-n} applied to $(\hat{\phi}^n,\tau^n)$ and $(\varepsilon^n,\mathrm{Id})$.
\end{proof}

Before the proof of Proposition \ref{prop:pointwise_char_decorated}, we make the following two basic observations about any given elements $x \in D[0,T]$ and $(x,I)\in \mathfrak{D}[0,T]$.

Firstly, for any $t\in[0,T]$ and $\varepsilon>0$, there exists $\delta$ such that
\begin{equation}\label{eq:1st_obs}
\mathrm{dist}(y,I_t) \leq \varepsilon \quad \text{for every}\quad y\in I_{t^\prime},\quad \text{for all}\quad t^\prime \in B(t,\delta).
\end{equation}
Indeed, Definition \ref{def:decorated_Skorokhod} gives $x(t)\in I_t$ and $d_\mathrm{H}(I_{s},\{x(t)\})\rightarrow 0$ as $s\downarrow t$, so, for $\delta >0$ small enough, $t^\prime\in (t,t+\delta)$ implies $\mathrm{dist}(y,I_t)\leq |y-x(t)|\leq \varepsilon $ for every $y\in I_{t^\prime}$, and the reasoning applies for $t^\prime\in (t-\delta ,t)$ by replacing $x(t)$ with $x(t-)$.

Secondly, for any $(s,z)\in \Gamma(\iota(x))$ and $\theta>0$, there exists $r\in[0,T]$ so that 
\begin{equation}\label{eq:2nd_obs}
	|r-s|\leq \theta \quad \text{and}\quad  |x(r)-z|\leq \theta. 
\end{equation}
Indeed, if $z=x(s)$, then $r=s$ works, and if $z=x(s-)$, then it follows from $x(r)\rightarrow x(s-)$ as $r\uparrow s$.
\begin{proof}[Proof of Proposition \ref{prop:pointwise_char_decorated}]
	Assume $d_\mathfrak{D}(x_n,(x,I))\rightarrow 0$, and fix $t\in[0,T]$ and $\varepsilon>0$. Let $\delta^\prime>0$ be as in \eqref{eq:1st_obs} given $\varepsilon^\prime>0$, where $\varepsilon^\prime=\varepsilon/2$, and set $\delta = \delta^\prime/2$. We can take $N$ so that (i) $\mathrm{dist}((s,z),\Gamma(I))\leq \delta \land \varepsilon^\prime$ for all $(s,z)\in \Gamma(\iota(x^n))$ and (ii) $\mathrm{dist}((s,y),\Gamma(\iota(x^n)))\leq \delta \land \varepsilon^\prime$ for all $(s,z)\in \Gamma (I)$.
	
	First, for any $s\in B(t,\delta)$, (i) implies that there exists $(s^\prime ,z^\prime)\in \Gamma (I)$ with $|s-s^\prime|,|x^n(s)-z^\prime|\leq  \delta \land \varepsilon^\prime $. Thus, the triangle inequality gives $|s^\prime - t| \leq \delta^\prime $, so \eqref{eq:1st_obs} implies $\mathrm{dist}(z^\prime , I_t)\leq \varepsilon^\prime$, and hence another triangle inequality gives $\mathrm{dist}(x^n(s),I_t) \leq \varepsilon$ for every $s\in B(t,\delta)$, for all $n>N$. Next, for any $z\in I_t$ and $n>N$, (ii) implies that there exist $(s_n,z_n)\in \Gamma(\iota(x^n))$ such that $|s_n-t|,|z_n-z|\leq \delta \land \varepsilon^\prime$. By \eqref{eq:2nd_obs} with $\theta =\delta \land \varepsilon^\prime$, we find $r_n\in[0,T]$ such that $|r_n-s_n|,|x^n(r_n)-z_n| \leq \delta \land \varepsilon^\prime$, so the triangle inequality gives $r_n \in B(t,\delta)$ with $|x^n(r^n)-z|\leq \varepsilon$, hence $\mathrm{dist}(z,x^n(B(t,\delta)))\leq \varepsilon$ for every $z\in I_t$, for all $n>N$. This verifies \eqref{eq:decorated_char}.
	
	Now assume instead that \eqref{eq:decorated_char} holds. We proceed by contradiction, so suppose that (i) $\sup_{(s,z)\in \Gamma (\iota (x^n))}\mathrm{dist}((s,z),\Gamma(I))\not \rightarrow 0$ or (ii) $\sup_{(t,y)\in \Gamma(I)}\mathrm{dist}((t,y),\Gamma(\iota(x^n)))\not \rightarrow 0$.
	
	In case (i), we can find $\varepsilon>0$, $s_n\rightarrow s_*$ in $[0,T]$, and $z_n \in \{x^n(s_n-),x^n(s_n)\}$ such that $\mathrm{dist}((s_n,z_n),\Gamma(I)) \geq \varepsilon$. For small $\delta \in (0,\varepsilon)$, \eqref{eq:decorated_char} at $t=s_*$ gives an $N>0$ so that $|s_n-s_*|\leq \delta /2 $ and $d_\mathrm{H}(x^n(B(s_*,\delta)),I_{s_*})\leq \varepsilon/4$ for all $n>N$. 
	By \eqref{eq:2nd_obs}, there exist $r_n$ with $|r_n-s_n|,|x^n(r_n)-z_n|\leq \delta/4$, so $r_n \in B(s_*,\delta)$ by the triangle inequality and hence $|x^n(r_n)-z_*|\leq \varepsilon /4$ for some $z_*\in I_{s_*}$. Thus, $(s_*,z_*)\in \Gamma(I)$ satisfies $|s_n-s_*|\leq \delta /2 \leq \varepsilon /2$ and $|z_n-z_*|\leq \varepsilon/2$ by the triangle inequality, contradicting $\mathrm{dist}((s_n,z_n),\Gamma(I)) \geq \varepsilon$.
	
	In case (ii), the compactness of $\Gamma(I)$ implies that we can find $\varepsilon>0$ and points $(t_n,y_n)\rightarrow (t_*,y_*)$ in $\Gamma (I)$ so that $\mathrm{dist}((t_n,z_n),\Gamma(\iota(x^n)))\geq \varepsilon$. Taking $\delta \in (0,\varepsilon)$ small enough,  \eqref{eq:decorated_char} gives an $N>0$ so that $\mathrm{dist}(y,x^n(B(t_*,\delta/4))) \leq \varepsilon/4$ for all $y\in I_{t_*}$ when $n>N$. Since $y_*\in I_{t_*}$, we can thus find $s_n\in B(t_*,\delta/4)$ with $|y_*-x^n(s_n)|\leq \varepsilon/4$ for $n>N$. With $N$ large enough, we also have $|t_n-t_*|,|y_n-y_*|\leq \delta /4$, so the triangle gives $|t_n-s_n|\leq \delta/2 \leq \varepsilon/2$ and $|y_n-x_n(s_n)|\leq \varepsilon/2$ which is again a contradiction.
\end{proof}

\appendix
\section{Dambis--Dubins--Schwarz theorem}
\label{app:DDS}
The following proposition is a version of the Dambis--Dubins--Schwarz theorem with enlargement of the probability space, which is specifically adapted to our setup.
\begin{prop}
    \label{prop:DDS}
    Let $\left( \Omega\,, \mathcal F\,, (\mathcal F_t)_{t \leq T}\,, \mathbb P\right)$ be a filtered probability space supporting two adapted processes $(X_t\,, M_t)_{t \leq T}$ such that $M$ is a continuous martingale with quadratic variation $\langle M\rangle_t =: A_t\,$. Then, there exists an extended probability space $\left(\widetilde \Omega\,, \mathcal G\,, (\mathcal G_t)_{t \geq 0}\,, \widetilde{\mathbb P}\right)$ supporting processes $\widetilde X\,$, $\widetilde M\,$, $\widetilde A\,$, and a $\left(\mathcal G_t\right)_{t \geq 0}$-Brownian motion $W\,$, such that 
    \begin{itemize}
        \item $\left(\widetilde X\,, \widetilde M\,, \widetilde A\right) \sim (X\,, M\,, A)\,$.
        \item For any $t \leq T\,$, $\widetilde{M}_t = W_{\widetilde{A}_t}\,$.
        \item For any $t \leq T\,$, $\widetilde{A}_t$ is a $(\mathcal G_s)_{s \geq 0}$-stopping time.
    \end{itemize}
\end{prop}
\begin{proof}
    We first define the stopping times 
    $$
    \tau_t := \inf\{s \geq 0\,:\, A_s > t\}\,, \quad t \geq 0\,.
    $$
    Let $(\Omega'\,, \mathcal F'\,, (\mathcal F'_t)_{t \geq 0}\,, \mathbb P')$ be a filtered probability supporting a Brownian motion $B\,$. We define 
    $$
    \widetilde \Omega := \Omega \times \Omega'\,, \quad \mathcal G := \mathcal F \otimes \mathcal F'\,,\quad \mathcal G_t := \mathcal F_{\tau_t} \otimes \mathcal F_t'\,,\quad \widetilde{\mathbb P} := \mathbb P \otimes \mathbb P'\,,
    $$
    as well as 
    $$
    \left(\widetilde X\,, \widetilde M\,, \widetilde A\,, \widetilde \tau \right)(\omega\,, \omega') := \left(X\,, M\,, A\,, \tau \right)(\omega)\,, \quad \widetilde B(\omega\,, \omega') := B(\omega')\,, \quad (\omega\,, \omega') \in \widetilde \Omega\,,
    $$
    so that the first point of the proposition is verified, and $\widetilde{B}$ is independent of $(\widetilde X\,, \widetilde{M}\,, \widetilde{A})\,$.

    We consider the $(\mathcal G_t)_{t\geq 0}$-adapted process
    $$
    W_t := \widetilde{M}_{\widetilde{\tau}_t} + \widetilde{B}_t - \widetilde{B}_{t \wedge\widetilde{A}_T}\,, \quad t \geq 0\,.
    $$
    Following the proof of \cite[Theorem~V.1.7]{revuzYor1999}, we obtain that $W$ is a $(\mathcal G_t)_{t \geq 0}$-Brownian motion. Moreover, $\widetilde{M}_{\widetilde{\tau}_{{\widetilde A}_t}} = \widetilde M_t$ for all $t \leq T\,$, since $\widetilde M$ and $\widetilde A$ have the same intervals of constancy, so that $\widetilde M_t = W_{\widetilde A_t}\,$. It remains to prove that for $t\leq T\,$, $\widetilde A_t$ is a $(\mathcal G_s)_{s \geq 0}$-stopping time. For $s\,, u \geq 0\,$,
    $$
    \{A_t \geq s\} \cap\{\tau_s \leq u\} = \{A_t \geq s\} \cap \{A_u \geq s\} = \{ A_{t \wedge u} \geq s\} \in \mathcal F_u\,.
    $$
    This being true for all $u \geq 0\,$, we conclude that $\{A_t \geq s\} \in \mathcal F_{\tau_s}$ for all $s \geq 0\,$. Hence, 
    $$
    \{\widetilde A_t \geq s\} = \{A_t \geq s\} \times \Omega' \in \mathcal F_{\tau_s} \times \mathcal F_s' = \mathcal G_s\,, \quad s \geq 0\,,
    $$
    which concludes the proof.
\end{proof}

\section{Convolution kernels and resolvents}
\label{app:resolvents}
We gather here some results on locally integrable convolution kernels and their resolvents. In the following, $f * g$ denotes the convolution product 
$$
(f * g)(t) := \int_0^t \, f(t-s)\,g(s)\, ds\,, \quad t \geq 0\,.
$$
For any $K \in L^1_{loc}(\R_+\,, \R)\,$, there exists a unique $R \in L^1_{loc}(\R_+\,, \R)\,$, called its \textit{resolvent of the second kind}, satisfying 
$$
K * R = R * K = K - R\,,
$$
see \cite[Theorem~2.3.1]{gripenberg1990volterra}. Classical examples include:
\begin{itemize}
    \item The exponential kernel $K(t) = c e^{bt}\,$, $c\,, b \in \R\,$, for which $R(t) = ce^{(b-c)t}\,$.
    \item The fractional kernel $K(t) = \frac{c}{\Gamma(\alpha)}t^{\alpha - 1}\,$, $c \in \R\,$, $\alpha > 0\,$, with $R(t) = c t^{\alpha - 1}E_{\alpha, \alpha}(-ct^{\alpha})\,$.
    \item The gamma kernel $K(t) = \frac{c}{\Gamma(\alpha)}e^{bt} t^{\alpha -1}\,$, $c\,, b \in \R\,$, $\alpha > 0\,$, with $R(t) = c e^{bt}t^{\alpha-1}E_{\alpha, \alpha}(-ct^{\alpha})\,$.
\end{itemize}
Here, $E_{\alpha, \beta}$ denotes the two-parameter Mittag-Leffler function.\\

In the case of a completely monotone kernel $K$ (see Definition \ref{defn:cm}), we have more information on its resolvent, such as its sign.

\begin{prop}
\label{prop:nonnegative_resolvent_cm}
    Let $K \in L^1_{loc}(\R_+\,, \R_+)$ be a completely monotone kernel. Then its resolvent $R$ is nonnegative, and 
    $$
    \int_0^{+ \infty}\,R(s)\,ds \leq 1\,.
    $$
\end{prop}
\begin{proof}
    Since $K$ is completely monotone, it is log-convex. To see this, one has to write $K$ as the Laplace transform of a nonnegative measure on $[0\,, \infty)$ thanks to the Bernstein theorem, and log-convexity follows from the Cauchy-Schwartz inequality. Hence, \cite[Theorem~9.8.1]{gripenberg1990volterra} proves that $R$ is nonnegative. Let us introduce the continuous non-decreasing functions $\bar R(t) := \int_0^t \,R(s)\,ds$ and $\bar K(t) := \int_0^t\,K(s)\,ds\,$. By definition, they satisfy
    $$
    \bar R(t) = \bar K(t) - (K * \bar R)(t)\,, \quad t \geq 0\,.
    $$
    Assume that there exists $t_0 > 0$ such that $\bar R(t_0) = 1\,$, then for any $t \geq t_0\,$,
    \begin{align*}
        \bar R(t) &= \bar K(t)- \int_0^{t_0}\,K(t_0-s)\,\bar R(s)\,ds + \int_0^{t_0}\,(K(t_0 -s) - K(t-s))\,\bar R(s)\,ds \\
        &\quad - \int_{t_0}^t\,K(t-s)\,\bar R(s)\,ds \\
        &= \bar K(t) + 1 - \bar K(t_0) + \int_0^{t_0}\,(K(t_0 -s) - K(t-s))\,\bar R(s)\,ds \\
        &\quad - \int_{t_0}^t\,K(t-s)\,\bar R(s)\,ds\,.
    \end{align*}
    since $K$ is non-increasing, and $\bar R$ is non-decreasing, we obtain 
    \begin{align*}
        \bar R(t) &\leq \bar K(t) + 1 - \bar K(t_0) + \int_0^{t_0} \,(K(t_0-s) - K(t-s))\,ds  - \int_{t_0}^t\,K(t-s)\,ds \\
        &= 1\,.
    \end{align*}
    Hence, we deduce that $\bar R(t) \leq 1$ for all $t \geq 0\,$, which concludes the proof.
\end{proof}

More generally, for a non-increasing kernel $K\,$, the resolvent of $-K$ is easier to deal with.
\begin{prop}
    \label{prop:nonpositive_resolvent}
    Let $K \in L^1_{loc}(\R_+\,, \R_+)$ be a non-increasing kernel. Then the resolvent of the second kind of $-K$ is nonpositive.
\end{prop}
\begin{proof}
    Since $K$ is locally integrable and non-increasing, any interval of the form $[T_1\,,T_2]\,$, with $0 \leq T_1 \leq T_2\,$, can be divided into finitely many subintervals $[t_{i}\,, t_{i+1}]$ such that $\|K\|_{L^1([t_i, t_{i+1}])} < 1\,$. Since $-K$ is nonpositive, we can apply \cite[Proposition~9.8.1]{gripenberg1990volterra}, proving that its resolvent is nonpositive.
\end{proof}

We conclude this section by proving two different types of scalings of convolution kernels leading to Dirac limits.

\begin{proof}[Proof of Lemma \ref{lemma:weak_cv_large_time_L1}]
    Since $K$ is integrable, its Laplace transform
    $$
    \hat K(\lambda) := \int_0^{\infty}\,K(s)\,e^{-\lambda s}\,ds\,, \quad \lambda \geq 0\,,
    $$
    is well-defined and continuous. Defining $K^n(t) := n K(nt)\,$, we easily compute 
    $$
    \hat K^n (\lambda) = \hat K(\lambda / n) \overset{n \to + \infty}{\longrightarrow} \hat K(0) = \|K\|_{L_1(\R_+)}\,, \quad \lambda \geq 0\,,
    $$
    which is the Laplace transform of the measure $\|K\|_{L^1(\R_+)}\,\delta_0(dt)\,$. From \cite[Theorem~XIII.1.2a]{feller1991introduction} we conclude that the convergence of the Laplace transforms leads to weak convergence on any $[0\,,T]$ with $T > 0\,$.
\end{proof}

\begin{proof}[Proof of Lemma \ref{lemma:weak_cv_resolvent_mean_rev}]
    Since $K$ is completely monotone, it is non-increasing, using the fact that it is locally integrable, we have 
    $$
    \widehat K(\lambda) := \int_0^{+ \infty}\,K(s)\,e^{-\lambda s}\,ds \leq \int_0^1\,K(s)\,ds + K(1)\,\int_1^{+\infty}\,e^{-\lambda s}\,ds < + \infty\,,
    $$
    for any $\lambda > 0\,$. This is also true for $R^n$ for any $n > 0\,$, since it is integrable on $\R_+$ thanks to Proposition \ref{prop:nonnegative_resolvent_cm}. Moreover, from the definition of $R^{n}\,$, we have the equality 
    $$
    \widehat R^{n}(\lambda) = \frac{ n\,\widehat K(\lambda)}{1 + n \, \widehat K(\lambda)}\,, \quad \lambda > 0\,,
    $$
    between their respective Laplace transforms. Hence, 
    $$
    \widehat R^{n}(\lambda) \overset{n \to + \infty}{\longrightarrow} 1\,, \quad \lambda > 0\,,
    $$
    which is the Laplace transform of the measure $ \delta_0(dt)\,$. From \cite[Theorem~XIII.1.2a]{feller1991introduction}, we conclude that the convergence of the Laplace transforms leads to a weak convergence on any $[0\,,T]$ with $T > 0\,$.
\end{proof}

\bibliographystyle{alpha}
\bibliography{ref}

@article {Meerschaert_2014,
    AUTHOR = {Meerschaert, Mark M. and Straka, Peter},
     TITLE = {Semi-{M}arkov approach to continuous time random walk limit
              processes},
   JOURNAL = {The Annals of Probability},
  FJOURNAL = {The Annals of Probability},
    VOLUME = {42},
      YEAR = {2014},
    NUMBER = {4},
     PAGES = {1699--1723},
}

@article {Kim_Song_Vondracek,
    AUTHOR = {Kim, Panki and Song, Renming and Vondra\v{c}ek, Zoran},
     TITLE = {On the potential theory of one-dimensional subordinate
              {B}rownian motions with continuous components},
   JOURNAL = {Potential Analysis},
  FJOURNAL = {Potential Analysis. An International Journal Devoted to the
              Interactions between Potential Theory, Probability Theory,
              Geometry and Functional Analysis},
    VOLUME = {33},
      YEAR = {2010},
    NUMBER = {2},
     PAGES = {153--173},
}

@book {Kyprianou,
    AUTHOR = {Kyprianou, Andreas E.},
     TITLE = {Fluctuations of {L}\'evy processes with applications},
    SERIES = {Universitext},
   EDITION = {Second},
      NOTE = {Introductory lectures},
 PUBLISHER = {Springer, Heidelberg},
      YEAR = {2014},
     PAGES = {xviii+455},
}

@article{CKM,
  title={Superdiffusive limits beyond the {M}arcus regime for deterministic fast-slow systems},
  author={Chevyrev, Ilya and Korepanov, Alexey and Melbourne, Ian},
  journal={Communications of the American Mathematical Society},
  volume={4},
  number={16},
  pages={746--786},
  year={2024}
}

@article{FFT,
	author    = {Freitas, Ana Cristina Moreira and Freitas, Jorge Milhazes and Todd, Mike},
	title     = {Enriched functional limit theorems for dynamical systems},
	journal   = {Annali della Scuola Normale Superiore di Pisa - Classe di Scienze},
	volume    = {59},
	year      = {2025},
	doi       = {10.2422/2036-2145.202401_016},
}

@article{FFMT,
	author    = {Freitas, Ana Cristina Moreira and Freitas, Jorge Milhazes and Melbourne, Ian and Todd, Mike},
	title     = {Convergence to decorated {L}\'evy processes in non-{S}korohod topologies for dynamical systems},
	journal   = {Electronic Journal of Probability},
	volume    = {29},
	year      = {2024},
	pages     = {1-24},
	doi       = {10.1214/24-EJP1231},
}

@article{jaber2024reconciling,
  title={Reconciling rough volatility with jumps},
  author={Abi Jaber, Eduardo and De Carvalho, Nathan},
  journal={SIAM Journal on Financial Mathematics},
  volume={15},
  number={3},
  pages={785--823},
  year={2024},
  publisher={SIAM}
}

@article{abi2025hyper,
  title={From Hyper Roughness to Jumps as {$H \to -1/2$}},
  author={Abi Jaber, Eduardo and Attal, Elie and Rosenbaum, Mathieu},
  year={2025},
  journal={Annals of Applied Probability},
  note={To appear, arXiv:2503.16985}
}

@article{abi2019affine,
  title={Affine {V}olterra processes},
  author={Abi Jaber, Eduardo and Larsson, Martin and Pulido, Sergio},
  year={2019},
  journal={Annals of Applied Probability},
  volume={29},
  number = {5},
  pages = {3155-3200}
}

@article{jusselin2020no,
  title={No-arbitrage implies power-law market impact and rough volatility},
  author={Jusselin, Paul and Rosenbaum, Mathieu},
  journal={Mathematical Finance},
  volume={30},
  number={4},
  pages={1309--1336},
  year={2020},
  publisher={Wiley Online Library}
}

@article{abi2021weak,
  title={Weak existence and uniqueness for affine stochastic {V}olterra equations with {$L^1$}-kernels},
  author={Abi Jaber, Eduardo},
  year={2021},
  journal={Bernoulli},
  volume={27},
  number = {3},
  pages = {1583-1615}
}

@book{jacod2013limit,
  title={Limit theorems for stochastic processes},
  author={Jacod, Jean and Shiryaev, Albert N.},
  series={Grundlehren der
mathematischen Wissenschaften},
  volume={288},
  year={2013},
  edition={2nd},
  publisher={Springer}
}

@article{jakubowski1997non,
  title={A non-{S}korohod topology on the {S}korohod space},
  author={Jakubowski, Adam},
  year={1997},
  journal={Electronic Journal of Probability},
  volume={2},
  number={4},
  pages={1--21}
}

@article{sojmark2023weak,
  title={Weak Convergence of Stochastic Integrals on {S}korokhod Space in {S}korokhod's {$J_1$} and {$M_1$} Topologies},
  author={S{\o}jmark, Andreas and Wunderlich, Fabrice},
  journal={Probability Theory and Related Fields},
  note = {https://doi.org/10.1007/s00440-026-01476-y},
  year={2026}
}

@book{whitt2002stochastic,
  author={Whitt, Ward},
  title={Stochastic-process limits: an introduction to stochastic-process limits and their application to queues},
  series={Springer Series in Operations Research},
  publisher={Springer},
  year={2002}
}

@article{abi2019markovian,
  title={Markovian structure of the {V}olterra {H}eston model},
  author={Abi Jaber, Eduardo and El Euch, Omar},
  journal={Statistics \& Probability Letters},
  volume={149},
  pages={63--72},
  year={2019},
  publisher={Elsevier}
}

@book{feller1991introduction,
  title={An introduction to probability theory and its applications, Volume 2},
  author={Feller, William},
  volume={2},
  year={1991},
  publisher={John Wiley \& Sons}
}

@book{revuzYor1999,
  author    = {Revuz, Daniel and Yor, Marc},
  title     = {Continuous Martingales and {B}rownian Motion},
  series    = {Grundlehren der mathematischen {W}issenschaften},
  volume    = {293},
  edition   = {3},
  publisher = {Springer},
  address   = {Berlin, Heidelberg},
  year      = {1999},
  doi       = {10.1007/978-3-662-06400-9}
}

@article{dawson_super-brownian_1994,
	title = {A super-{Brownian} motion with a single point catalyst},
	volume = {49},
	issn = {0304-4149},
	url = {https://www.sciencedirect.com/science/article/pii/0304414994901104},
	doi = {10.1016/0304-4149(94)90110-4},
	number = {1},
	urldate = {2026-02-23},
	journal = {Stochastic Processes and their Applications},
	author = {Dawson, Donald A. and Fleischmann, Klaus},
	month = jan,
	year = {1994},
	pages = {3--40},
}

@article{zahle_space-time_2005,
	title = {Space-time regularity of catalytic super-{Brownian} motion},
	volume = {278},
	issn = {1522-2616},
	url = {https://onlinelibrary.wiley.com/doi/10.1002/mana.200310284},
	doi = {10.1002/mana.200310284},
	language = {en},
	number = {7-8},
	urldate = {2026-02-23},
	journal = {Mathematische Nachrichten},
	author = {Z\"ahle, Henryk},
	month = jun,
	year = {2005},
	note = {Publisher: John Wiley \& Sons, Ltd},
	pages = {942--970},
}

@article{klenke_absolute_2000,
	title = {Absolute continuity of catalytic measure-valued branching processes},
	volume = {89},
	copyright = {https://www.elsevier.com/tdm/userlicense/1.0/},
	issn = {03044149},
	url = {https://linkinghub.elsevier.com/retrieve/pii/S0304414900000223},
	doi = {10.1016/S0304-4149(00)00022-3},
	language = {en},
	number = {2},
	urldate = {2026-02-24},
	journal = {Stochastic Processes and their Applications},
	author = {Klenke, Achim},
	month = oct,
	year = {2000},
	pages = {227--237},
}

@article{dawson_critical_1991,
	title = {Critical branching in a highly fluctuating random medium},
	volume = {90},
	issn = {1432-2064},
	url = {https://doi.org/10.1007/BF01192164},
	doi = {10.1007/BF01192164},
	language = {en},
	number = {2},
	urldate = {2026-02-24},
	journal = {Probability Theory and Related Fields},
	author = {Dawson, Donald A. and Fleischmann, Klaus},
	month = jun,
	year = {1991},
	pages = {241--274},
}

@article{wang_fluctuation_2014,
	title = {Fluctuation limits of the single point catalytic super-stable processes},
	volume = {30},
	issn = {1439-7617},
	url = {https://doi.org/10.1007/s10114-013-2465-9},
	doi = {10.1007/s10114-013-2465-9},
	language = {en},
	number = {1},
	urldate = {2026-02-24},
	journal = {Acta Mathematica Sinica, English Series},
	author = {Wang, Li},
	year = {2014},
	pages = {23--34},
}

@article{mytnik_uniqueness_2015,
  title={Uniqueness for {V}olterra-type stochastic integral equations},
  author={Mytnik, Leonid and Salisbury, Thomas S.},
  journal={arXiv preprint arXiv:1502.05513},
  year={2015}
}

@article{fleischmann1995new,
  title={A new approach to the single point catalytic super-{B}rownian motion},
  author={Fleischmann, Klaus and Le Gall, Jean-Fran{\c{c}}ois},
  journal={Probability theory and related fields},
  volume={102},
  number={1},
  pages={63--82},
  year={1995},
  publisher={Springer}
}

@article{dawson1995singularity,
  title={Singularity of super-{B}rownian local time at a point catalyst},
  author={Dawson, Donald A. and Fleischmann, Klaus and Li, Yi and Mueller, Carl},
  journal={The Annals of Probability},
  volume={23},
  number={1},
  pages={37--55},
  year={1995},
  publisher={JSTOR}
}

@article{el2019characteristic,
  title={The characteristic function of rough {H}eston models},
  author={El Euch, Omar and Rosenbaum, Mathieu},
  journal={Mathematical Finance},
  volume={29},
  number={1},
  pages={3--38},
  year={2019},
  publisher={Wiley Online Library}
}

@article{jaisson2016rough,
  title={Rough fractional diffusions as scaling limits of nearly unstable heavy tailed {H}awkes processes},
  author={Jaisson, Thibault and Rosenbaum, Mathieu},
  year={2016},
  journal={Annals of Applied Probability},
  volume={26},
  number={5},
  pages={2860-2882}
}

@book{gripenberg1990volterra,
  title={Volterra integral and functional equations},
  author={Gripenberg, Gustaf and Londen, Stig-Olof and Staffans, Olof},
  year={1990},
  publisher={Cambridge University Press}
}

@article{bachelierTheorieSpeculation1900,
  title = {{Th\'eorie de la sp\'eculation}},
  author = {Bachelier, Louis},
  year = 1900,
  journal = {Annales scientifiques de l'\'Ecole normale sup\'erieure},
  volume = {17},
  pages = {21--86},
  issn = {0012-9593, 1873-2151},
  doi = {10.24033/asens.476},
  urldate = {2026-03-25},
  langid = {french}
}

@article{schrodingerZurTheorieFall1915,
  title = {Zur {{Theorie}} Der {{Fall-}} Und {{Steigversuche}} an {{Teilchen}} Mit {{Brownscher Bewegung}}},
  author = {Schr{\"o}dinger, E.},
  year = 1915,
  journal = {Physikalische Zeitschrift},
  volume = {16},
  pages = {289--295}
}

@incollection{breimanFirstExitTimes1967a,
  title = {First Exit Times from a Square Root Boundary},
  booktitle = {Proceedings of the {{Fifth Berkeley Symposium}} on {{Mathematical Statistics}} and {{Probability}}, {{Volume}} 2: {{Contributions}} to {{Probability Theory}}, {{Part}} 2},
  author = {Breiman, Leo},
  year = 1967,
  month = jan,
  volume = {5.2B},
  pages = {9--17},
  publisher = {University of California Press},
  urldate = {2026-03-25},
  langid = {english}
}

@incollection{feller1951diffusion,
  title     = {Diffusion processes in genetics},
  author    = {Feller, William},
  booktitle = {Proceedings of the Second {Berkeley} Symposium on Mathematical Statistics and Probability},
  pages     = {227--247},
  year      = {1951},
  publisher = {University of California Press},
  address   = {Berkeley},
}

@article{sheppFirstPassageProblem1967a,
  title = {A first passage problem for the {W}iener process},
  author = {Shepp, L. A.},
  year = 1967,
  journal = {The Annals of Mathematical Statistics},
  volume = {38},
  number = {6},
  pages = {1912--1914},
  publisher = {Institute of Mathematical Statistics},
  issn = {0003-4851, 2168-8990},
  doi = {10.1214/aoms/1177698626},
  urldate = {2026-03-25},
  langid = {english}
}

@article{novikovStoppingTimesWiener1971,
  title = {On Stopping Times for a {W}iener Process},
  author = {Novikov, A. A.},
  year = {1971},
  journal = {Theory of Probability \& Its Applications},
  volume = {16},
  number = {3},
  pages = {449--456},
  issn = {0040-585X, 1095-7219},
  doi = {10.1137/1116049},
  urldate = {2026-02-23},
  langid = {english}
}

@article{fortetFonctionsAleatoiresType1943,
  title={Les fonctions al{\'e}atoires du type de {M}arkoff associ{\'e}es {\`a} certaines {\'e}quations lin{\'e}aires aux d{\'e}riv{\'e}es partielles du type parabolique},
  author={Fortet, Robert},
  journal={Journal de Math{\'e}matiques Pures et Appliqu{\'e}es},
  volume={22},
  pages={177--243},
  year={1943}
}

@article{durbinFirstPassageDensityContinuous1985a,
  title = {The First-Passage Density of a Continuous {G}aussian Process to a General Boundary},
  author = {Durbin, J.},
  year = 1985,
  journal = {Journal of Applied Probability},
  volume = {22},
  number = {1},
  eprint = {3213751},
  eprinttype = {jstor},
  pages = {99--122},
  publisher = {Applied Probability Trust},
  issn = {0021-9002},
  doi = {10.2307/3213751},
  urldate = {2026-03-25}
}

@article{durbinFirstPassageDensityBrownian1992a,
  title = {The First-Passage Density of the {B}rownian Motion Process to a Curved Boundary},
  author = {Durbin, J. and Williams, D.},
  year = 1992,
  journal = {Journal of Applied Probability},
  volume = {29},
  number = {2},
  eprint = {3214567},
  eprinttype = {jstor},
  pages = {291--304},
  publisher = {Applied Probability Trust},
  issn = {0021-9002},
  doi = {10.2307/3214567},
  urldate = {2026-03-25}
}

@article{peskirIntegralEquationsArising2002,
  title = {On integral equations arising in the first-passage problem for {B}rownian motion},
  author = {Peskir, Goran},
  year = 2002,
  journal = {Journal of Integral Equations and Applications},
  volume = {14},
  number = {4},
  pages={397-423},
  issn = {0897-3962},
  doi = {10.1216/jiea/1181074930},
  urldate = {2026-03-24}
}

@incollection{strassenAlmostSureBehavior1967a,
  title = {Almost Sure Behavior of Sums of Independent Random Variables and Martingales},
  booktitle = {Proceedings of the {{Fifth Berkeley Symposium}} on {{Mathematical Statistics}} and {{Probability}}, {{Volume}} 2: {{Contributions}} to {{Probability Theory}}, {{Part}} 1},
  author = {Strassen, Volker},
  year = 1967,
  month = jan,
  volume = {5.2A},
  pages = {315--344},
  publisher = {University of California Press},
  urldate = {2026-03-25},
  langid = {english}
}

@article{danielsMaximumSizeClosed1974a,
  title = {The Maximum Size of a Closed Epidemic},
  author = {Daniels, H. E.},
  year = 1974,
  journal = {Advances in Applied Probability},
  volume = {6},
  number = {4},
  eprint = {1426182},
  eprinttype = {jstor},
  pages = {607--621},
  publisher = {Applied Probability Trust},
  issn = {0001-8678},
  doi = {10.2307/1426182},
  urldate = {2026-03-25}
}

@article{ferebeeTangentApproximationOnesided1982a,
  title = {The Tangent Approximation to One-Sided {{Brownian}} Exit Densities},
  author = {Ferebee, Brooks},
  year = 1982,
  month = sep,
  journal = {Zeitschrift f\"ur Wahrscheinlichkeitstheorie und Verwandte Gebiete},
  volume = {61},
  number = {3},
  pages = {309--326},
  issn = {1432-2064},
  doi = {10.1007/BF00539832},
  urldate = {2026-03-25},
  langid = {english}
}

@book{lercheBoundaryCrossingBrownian1986,
  author    = {Lerche, Hans Rudolf},
  title     = {Boundary crossing of {B}rownian motion},
  subtitle  = {Its Relation to the Law of the Iterated Logarithm and to Sequential Analysis},
  series    = {Lecture Notes in Statistics},
  volume    = {40},
  publisher = {Springer},
  address   = {New York},
  year      = {1986},
  doi       = {10.1007/978-1-4615-6569-7}
}

@article{mechkovFastReversionLimitHeston2014,
  title={Fast-reversion limit of the {H}eston model},
  author={Mechkov, Serguei},
  journal={Available at SSRN 2418631},
  year={2015}
}

@article{bondiLevyProcessesWeak2025,
  title = {L\'evy Processes as Weak Limits of Rough {{Heston}} Models},
  author = {Bondi, Alessandro and Forde, Martin},
  year = 2025,
  journal = {arXiv preprint arXiv:2508.14835},
}

@article{fordeLARGEMATURITYSMILEHESTON,
  title = {The large-maturity smile for the {H}eston model},
  author = {Forde, Martin and Jacquier, Antoine},
  journal={Finance and Stochastics},
  volume={15},
  number={4},
  pages = {755-780},
  year = {2011}
}

@article{jaber2025simulating,
  title={Simulating integrated {V}olterra square-root processes and {V}olterra {H}eston models via {I}nverse {G}aussian},
  author={Abi Jaber, Eduardo and Attal, Elie},
  journal={arXiv preprint arXiv:2504.19885},
  year={2025}
}

@phdthesis{McCrickerd2021,
  author       = {Ryan McCrickerd},
  title        = {On spatially irregular ordinary differential equations and a pathwise volatility modelling framework},
  school       = {Imperial College London},
  year         = {2021},
  type         = {Ph.D. thesis},
  doi          = {10.25560/92202},
  url          = {https://doi.org/10.25560/92202}
}

@article{berger1980volterra,
  title={Volterra equations with {I}t{\^o} integrals---{I}},
  author={Berger, Marc A. and Mizel, Victor J.},
  journal={The Journal of Integral Equations},
  volume={2},
  number={3},
  pages={187--245},
  year={1980},
  publisher={JSTOR}
}

@article{abi2021weaksolution,
  title={A weak solution theory for stochastic {V}olterra equations of convolution type},
  author={Abi Jaber, Eduardo and Cuchiero, Christa and Larsson, Martin and Pulido, Sergio},
  journal={The Annals of Applied Probability},
  volume={31},
  number={6},
  pages={2924--2952},
  year={2021},
  publisher={Institute of Mathematical Statistics}
}

@article{friesen2024volterra,
  title={Volterra square-root process: stationarity and regularity of the law},
  author={Friesen, Martin and Jin, Peng},
  journal={The Annals of Applied Probability},
  volume={34},
  number={1A},
  pages={318--356},
  year={2024},
  publisher={Institute of Mathematical Statistics}
}

\end{document}